\documentclass[11pt]{article}
\usepackage[english]{babel}
\usepackage{times}
\usepackage{amsmath,amstext,amssymb,amsthm, amsfonts}
\usepackage[dvips]{graphics,color}
\newcommand{\bi}{\begin{itemize}}  
\newcommand{\ei}{\end{itemize}}     
\newcommand{\bc}{\begin{center}}  
\newcommand{\ec}{\end{center}}     

\newcommand{\ls}[1]
   {\dimen0=\fontdimen6\the\font \lineskip=#1\dimen0
   \advance\lineskip.5\fontdimen5\the\font \advance\lineskip-\dimen0
   \lineskiplimit=.9\lineskip \baselineskip=\lineskip
   \advance\baselineskip\dimen0 \normallineskip\lineskip
   \normallineskiplimit\lineskiplimit \normalbaselineskip\baselineskip
   \ignorespaces }

\numberwithin{equation}{section}

\newcommand{\slim} {\mathop{\rm lim\,sup}}
\newcommand{\ilim} {\mathop{\rm lim\,inf}}

\newtheorem{lemma}{Lemma}[section]
\newtheorem{theorem}[lemma]{Theorem}
\newtheorem{corollary}[lemma]{Corollary}

\def\X{\mathbb{X}}

\def\A{\mathbb{A}}
\def\H{\mathbb{H}}
\def\R{\mathbb{R}}
\def\P{\mathbb{P}}
\def\F{\mathbb{F}}

\title{Average-Cost Markov Decision Processes with Weakly Continuous Transition Probabilities}

\begin{document}


\date{}
\maketitle

\begin{center}
  Eugene~A.~Feinberg \footnote{Department of Applied Mathematics and
Statistics,
 Stony Brook University,
Stony Brook, NY 11794-3600, USA, eugene.feinberg@sunysb.edu}, \
Pavlo~O.~Kasyanov\footnote{Institute for Applied System Analysis,
National Technical University of Ukraine ``Kyiv Polytechnic
Institute'', Peremogy ave., 37, build, 35, 03056, Kyiv, Ukraine,\
kasyanov@i.ua.}, and Nina~V.~Zadoianchuk\footnote{Institute for
Applied System Analysis, National Technical University of Ukraine
``Kyiv Polytechnic Institute'', Peremogy ave., 37, build, 35,
03056, Kyiv, Ukraine,\
ninellll@i.ua.} \\

\bigskip
\end{center}

\begin{abstract}
This paper presents sufficient conditions for the existence of
stationary optimal policies for average-cost Markov Decision
Processes with Borel state and action sets and with weakly
continuous transition probabilities. The one-step cost functions
may be unbounded, and action sets may be noncompact. The main
contributions of this paper are: (i) general sufficient conditions
for the existence of stationary discount-optimal and average-cost
optimal policies and descriptions of properties of value functions
and sets of optimal actions, (ii) a sufficient condition for the
average-cost optimality of a stationary policy in the form of
optimality inequalities, and (iii) approximations of average-cost
optimal actions by discount-optimal actions.
\end{abstract}

\sloppy \large

\section{Introduction}\label{s1}
This paper provides sufficient conditions for the existence of
stationary optimal policies for average-cost Markov Decision
Processes (MDPs) with Borel state and action sets and with weakly
continuous transition probabilities. The cost functions may be
unbounded and action sets may be noncompact. The main
contributions of this paper are: (i) general sufficient conditions
for the existence of stationary discount-optimal and average-cost
optimal policies and descriptions of properties of value functions
and sets of optimal actions (Theorems~\ref{Prop1}, \ref{teor1},
and \ref{teor3}), (ii) a new sufficient condition of average-cost
optimality based on optimality inequalities
(Theorem~\ref{prop:dcoe}), and (iii) approximations of
average-cost optimal actions by
discount-optimal actions (Theorem~\ref{teorapp}).  

 For infinite-horizon MDPs there are two major criteria: average costs per unit time and expected total discounted costs.
 The former is typically more difficult to analyze. The so-called
 vanishing discount factor approach
 is often used to approximate average costs per unit time by  normalized expected total discounted costs.  The literature on
 average-cost MDPs is vast.  Most of the earlier results are
 surveyed in Arapostathis et al.~\cite{Arap}.  Here we mention
 just a few references.

 For finite state and action sets, Derman \cite{Der} proved the
 existence of stationary average-cost optimal policies.  This
 result follows from Blackwell~\cite{Blac} and it also was
 independently proved by Viskov and Shiryaev~\cite{VS}. When either the state set or the action set is infinite,
 even $\epsilon$-optimal policies may not exist for some $\epsilon>0$;
 Ross~\cite{Ross71}, Dynkin and Yushkevich~\cite[Chapter 7]{DY}, Feinberg~\cite[Section 5]{Fein1980}.
For a finite state set and compact action sets,  optimal policies
may not exist; Bather~\cite{Bath},
 Chitashvili~\cite{Chit}, Dynkin and Yushkevich~\cite[Chapter 7]{DY}.

For MDP with finite state and action sets, there exist stationary
policies satisfying optimality equations (see Dynkin and
Yushkevich~\cite[Chapter 7]{DY}, where these equations are called
canonical), and, furthermore, any stationary policy satisfying
optimality equations is optimal. The latter is also true for MDPs
with Borel state and an action sets, if the value and weight (also
called bias) functions are bounded; Dynkin and
Yushkevich~\cite[Chapter 7]{DY}.   When the optimal value of
average
 costs per unit time does not depend on the initial state (the
 optimal value function is constant), the pair of optimality
 equations becomes a single equation. For bounded one-step
 costs, Taylor~\cite{Taylor}, Ross~\cite{Ross68} for a countable state space and  Ross~\cite{Ross68a}, Gubenko and
 Statland~\cite{GS} for a Borel state space provided  sufficient conditions for the
 validity of optimality
 equations with a bounded bias function; see also Dynkin and Yushkevich~\cite[Chapter
 7]{DY}. Under all known sufficient conditions for the existence of average-cost optimal policies for
 infinite-state MDPs, the value function is constant.


In many applications of infinite-state MDPs, one-step costs are
unbounded from above.  For example, holding costs may be unbounded
in queueing and inventory systems.  Sennott~\cite{Senb, Sens} (and
references therein) developed a theory for countable-state
problems with unbounded one-step costs. For unbounded costs,
optimality inequalities are used instead of optimality equations
to construct a stationary average-cost optimal policy.
Cavazos-Cadena~\cite{CC} provided an example, when optimality
inequalities hold while optimality equations do not.

Sch\"al~\cite{Schal} developed a theory for Borel state spaces and
compact action sets.  Two types of continuity assumptions for
transition probabilities are considered in Sch\"al~\cite{Schal}:
the setwise and weak continuity.  For a countable state space
these assumptions  coincide; see Chen and
Feinberg~\cite[Appendix]{CF}. Setwise convergence of probability
measures is stronger than weak convergence;  Hern\'andez-Lerma and
Lasserre~\cite[p. 186]{HLerma1}. Formally speaking, the setwise
continuity assumption for MDPs is not stronger than the weak
continuity assumption, since the former claims that the transition
probabilities are continuous in actions, while they are jointly
continuous in states and actions in the latter. However, the joint
continuity  of transition probabilities in states and actions
often holds in applications. For example, for inventory control
problems with uncountable state spaces, setwise continuity of
transition probabilities takes place if demand is a continuous
random variable, while weak continuity holds for arbitrarily
distributed demand; see Feinberg and Lewis~\cite[Section 4]{FL}.
The importance of weak convergence for practical applications is
mentioned in Hern\'andez-Lerma and Lasserre~\cite[p.
141]{HLL2000}.

In many applications action sets are not compact.
Hern\'andez-Lerma~\cite{HLerma} extended Sch\"al's~\cite{Schal}
results under the setwise continuity assumptions to possibly
noncompact action sets.   Sch\"al's~\cite{Schal} assumptions on
compactness of action sets and  lower semi-continuity of cost
functions in the action argument are replaced in
Hern\'andez-Lerma~\cite{HLerma} by a more general assumption,
namely, that the cost functions are inf-compact in the action
argument.  For weakly continuous transition probabilities and
possibly  noncompact action sets, Feinberg and Lewis~\cite{FL}
proved  the existence of stationary optimal policies for MDPs with
cost functions being inf-compact in both state and action
arguments when, in addition to Sch\"al's~\cite{Schal} boundness
assumption on the relative discounted value at each state, the
so-called local boundness condition was assumed.

The original goal of this study was to show that the results from
Feinberg and Lewis~\cite{FL} hold without local boundness
condition. However, the results of this paper are more general.
This paper provides a weaker boundness condition on the relative
discounted value (Assumption~(${\rm \bf \underline{B}}$) in
Section~\ref{s5}) than Assumption~(${\rm \bf B}$) introduced in
Sch\"al~\cite{Schal}. It also provides a more general and natural
assumption (Assumption~(${\rm \bf W^*}$) in Section~\ref{s3}) than
inf-compactness of the one-step cost function in both arguments.
The main result of this paper, Theorem~\ref{teor1}, establishes
the validity of optimality inequalities and the existence of
stationary optimal policies under Assumptions~(${\rm \bf W^*}$)
and (${\rm \bf \underline{B}}$).

While inf-compactness of the cost function in the action parameter
is a natural assumption, inf-compactness in the state argument is
a more restrictive condition. For example,  when the state space
is unbounded (e.g., the set of nonnegative numbers) and action
sets are compact,  the assumption, that the cost function is
inf-compact in both arguments, does not cover the case of bounded
costs functions studied by Ross~\cite{Ross68a}, Gubenko and
Shtatland~\cite{GS}, and Dynkin and Yushkevich~\cite[Chapter
7]{DY}. Assumption (${\rm \bf W^*}$) covers this case as well as
unbounded costs and noncompact action sets.

As follows from the example presented in  Luque-V\'asquez and
Hern\'andez-Lerma (1995), MDPs with lower-semicontinuous cost
functions may possess pathological properties, even if the
one-step cost function is inf-compact in the action variable.
Assumption (${\rm \bf W^*}$)(ii) removes this difficulty.  As
stated in Lemma~\ref{lm0}, this assumption is weaker than
Sch\"al's~\cite{Schal} compactness and continuity assumptions for
weakly continuous transition probabilities and than
inf-compactness of one-step cost functions in both arguments
(state and action)  assumed in Feinberg and
Lewis~\cite{FL}.  

%
%
%
%
%

\section{Model Description}\label{s2}
For a metric space $S$, let ${\mathcal B}(S)$ be a Borel
$\sigma$-field on $S$, that is, the $\sigma$-field generated by all
open sets of metric space $S$.  For a set $E\subset S$, we denote by
${\mathcal B}(E)$ the $\sigma$-field whose elements are
intersections of $E$ with elements of ${\mathcal B}(S)$.  Observe
that $E$ is a metric space with the same metric as on $S$, and
${\mathcal B}(E)$ is its Borel $\sigma$-field. For a metric space
$S$, we denote by $\P(S)$ the set of probability measures on
$(S,{\mathcal B}(S)).$ A sequence of probability measures
$\{\mu_n\}$ from $\P(S)$ converges weakly to $\mu\in\P(S)$ if for
any bounded continuous function $f$ on $S$
\[\int_S f(s)\mu_n(ds)\to \int_S f(s)\mu(ds) \qquad {\rm as \
}n\to\infty.
\]

Consider a discrete-time MDP with a \textit{state space} $\X$, an
\textit{action space} $\mathbb{A},$ one-step costs $c$, and
transition pobabilities $q$. Assume that $\X$ and $\mathbb{A}$ are
\textit{Borel subsets} of Polish (complete separable metric)
spaces with the corresponding metrics $\rho$ and $\gamma$.  For
all $x\in \X$ a nonempty Borel subset $A(x)$ of $\mathbb{A}$
represents the \textit{set of actions} available at $x.$ Define
the graph of ${A}$ by
$$
{\rm Gr} ({A})=\{(x,a) \, : \,  x\in \X, a\in A(x)\}.
$$
Assume also that

(i) ${\rm Gr} ({A})$ is a measurable subset of $\X\times
\mathbb{A}$, that is, ${\rm Gr}(A)\in {\mathcal B}({\rm Gr}(A))$,
where ${\mathcal B}({\rm Gr}(A))={\mathcal B} (\X)\otimes
{\mathcal B}(\mathbb{A})$;

(ii)  there exists a measurable mapping $\phi:\X\to \mathbb{A}$
such that $\phi(x)\in A(x)$ for all $x\in \X;$

The \textit{one step cost}, $c(x,a)\le +\infty,$ for choosing an
action $a\in A(x)$ in a state $x\in\X,$ is a \textit{bounded below
measurable} function on ${\rm Gr}({A}).$ Let $q(B|x,a)$  be the
\textit{transition kernel} representing the probability that the
next state is in $B\in {\mathcal B}(\X)$, given that the action
$a$ is chosen in the state $x$. This means that:

$\bullet$ $q(\cdot|x,a)$ is a probability measure on
$(\X,{\mathcal B}(\X))$ for all $ (x,a)\in \X\times \mathbb{A}$;

$\bullet$ $q(B|\cdot,\cdot)$ is a Borel function on $({\rm
Gr}({A}),{\mathcal B}({\rm Gr}({A})))$ for all $B\in {\mathcal
B}(\X)$.

\textit{The decision process proceeds as follows}:

$\bullet$ at each time epoch $n=0,1,...$ the current state
$x\in\X$   is observed;

$\bullet$ a decision-maker chooses an action $a\in A(x);$

$\bullet$ the cost $c(x,a)$ is incurred;

$\bullet$ the system moves to the next state according to the
probability law $q(\cdot|x,a).$

\noindent As explained in the text following the proof of
Lemma~\ref{lm00101}, if for each $x\in\X$ there exists $a\in A(x)$
with $c(x,a)<\infty,$ the measurability of ${\rm Gr}(A)$ and
inf-compactness of the cost function $c$ in the action variable
$a$ assumed later imply that assumption (ii) holds.

Let $\mathbb{H}_n=(\X\times \mathbb{A})^n\times \X$ be the
\textit{set of histories} by time $n=0,1,...$ and $ {\mathcal
B}(\mathbb{H}_n)=({\mathcal B}(\X)\otimes {\mathcal
B}(\mathbb{A}))^n\otimes {\mathcal B}(\X)$. A \textit{randomized
decision rule} at epoch $n=0,1,...$ is a regular transition
probability $\pi_n:H_n\to \mathbb{A}$ concentrated on $A(\xi_n)$,
that is, (i) $\pi_n(\cdot\,|\, h_n)$ is a  probability on
$(\A,\mathcal{B}({\A})),$ given the history $h_n=(\xi_0, u_0, \xi_1,
u_1, ..., u_{n-1}, \xi_n)\in\H_n$,  satisfying
$\pi_n(A(\xi_n)|h_n)=1$, and (ii) for all $B\in {\mathcal B}({A})$,
the function $\pi_n(B|\cdot)$ is Borel on $(\mathbb{H}_n,{\mathcal
B}(\mathbb{H}_n)).$ A \textit{ policy} is a sequence
$\pi=\{\pi_n\}_{n= 0,1,\ldots}$ of decision rules. Moreover, $\pi$
is called \textit{nonrandomized}, if each probability measure
$\pi_n(\cdot|h_n)$ is concentrated at one point. A nonrandomized
policy is called \textit{Markov}, if all of the decisions depend
 on the current state and time only. A Markov policy is called
\textit{stationary}, if all the decisions depend on the current
state only. Thus, a Markov policy $\phi$ is defined by a sequence
$\phi_0, \phi_1,\ldots$ of Borel mappings $\phi_n:\X\to\mathbb{A}$
such that $\phi_n(x)\in A(x)$ for all $x\in\X$. A stationary policy
$\phi$ is defined by a Borel mapping $\phi:\X\to\mathbb{A}$ such
that $\phi(x)\in A(x)$ for all $x\in\X$. Let
\[\mathbb{F}=\{\phi:\X\to \mathbb{A} \, : \,   \phi \mbox{ is
Borel and\ } \phi(x)\in A(x)\ {\rm for\ all\ } x\in \X\}\] be the
\textit{set of stationary policies}.

The Ionescu Tulcea theorem (Bertsekas and Shreve \cite[pp.
140-141]{Bert} or Hern\'andez-Lerma and Lassere
\cite[p.178]{HLerma1}) implies that an initial state $x$ and a
policy $\pi$ define a unique probability $P_x^{\pi}$ on the set of
all trajectories $\mathbb{H}_{\infty}=(\X\times
\mathbb{A})^{\infty}$ endowed with the product of $\sigma$-field
defined by Borel $\sigma$-field of $\X$ and $\mathbb{A}.$ Let
$\mathbb{E}_x^{\pi}$ be an expectation with respect to
$P_x^{\pi}$.

For a finite horizon $N=0,1,...,$ let us define the \textit{expected
total discounted costs}
\begin{equation}\label{eq1}
v_{N,\alpha}^{\pi}:=\mathbb{E}_x^{\pi}\sum\limits_{n=0}^{N-1}\alpha^nc(\xi_n,u_n),\qquad\qquad
x\in\X,
\end{equation}
where $\alpha\ge 0$ is the discount factor and
$v_{0,\alpha}^{\pi}(x)=0.$ When $\alpha=1$, we shall write
$v_N^\pi(x)$ instead of $v_{N,1}^\pi(x).$ When $N=\infty$ and
$\alpha\in [0,1)$,  (\ref{eq1}) defines an \textit{infinite horizon
expected total discounted cost}  denoted by $v_\alpha^\pi(x).$

The \textit{average cost per unit time} is defined as
\begin{equation}\label{eq2}
w^{\pi}(x):=\slim\limits_{N\to +\infty}\frac{1}{N}
v_N^\pi(x),\qquad\qquad\ \
x\in\X.\end{equation} 
For any function $g^{\pi}(x)$, including
$g^{\pi}(x)=v_{N,\alpha}^{\pi}(x)$,
$g^{\pi}(x)=v_{\alpha}^{\pi}(x)$, and $g^{\pi}(x)=w^{\pi}(x)$,
define the \textit{optimal cost} \begin{equation*}
g(x):=\inf\limits_{\pi\in \Pi}g^{\pi}(x), \qquad\qquad\qquad\quad\
\ x\in\X,
\end{equation*} where $\Pi$ is \textit{the set of all policies}.

A policy $\pi$ is called \textit{optimal} for the respective
criterion, if $g^{\pi}(x)=g(x)$ for all $x\in \X.$  For
$g^\pi=v_{n,\alpha}^\pi$, the optimal policy is called
\emph{$n$-horizon discount-optimal}; for $g^\pi=v_{\alpha}^\pi$,
it is called \emph{discount-optimal};  for $g^\pi=w^\pi$, it is
called \emph{average-cost optimal.}

It is well known (see, e.g., Bertsekas and Shreve
\cite[Proposition 8.2]{Bert}) that the functions $v_{n,\alpha}(x)$
recursively satisfy the following \textit{optimality equations}
with $v_{0,\alpha}(x)=0$ for all $x\in \X$,
\begin{equation}\label{eq3_3}
v_{n+1,\alpha}(x)=\inf\limits_{a\in A(x)}\left\{c(x,a)+\alpha
\int_\X v_{n,\alpha}(y)q(dy|x,a)\right\},\quad x\in
\mathbb{X},\,\,n=0,1,...\ .
\end{equation}
In addition, a Markov policy $\phi,$ defined at the first $N$
steps by the mappings $\phi_0,...\phi_{N-1},$ that satisfy for all
$n=1,...,N$ the equations
\begin{equation}\label{eq4_4}
v_{n,\alpha}(x)=c(x,\phi_{N-n}(x))+\alpha\int_\X
v_{n-1,\alpha}(y)q(dy|x,\phi_{N-n}(x)),\quad x\in \mathbb{X},
\end{equation}
is optimal for the horizon $N;$ see e.g. Bertsekas and Shreve
\cite[Lemma 8.7]{Bert}.

It is also well known (Bertsekas and Shreve \cite[Propositions 9.8
and 9.12]{Bert}) that $v_{\alpha}$, where $\alpha\in (0,1]$,
satisfies the following discounted cost optimality equation
(DCOE):
\begin{equation}\label{eq5}
v_{\alpha}(x)=\inf\limits_{a\in A(x)}\left\{c(x,a)+\alpha \int_\X
v_{\alpha}(y)q(dy|x,a)\right\},\quad x\in \X,
\end{equation}
and a stationary policy $\phi_\alpha$ is discount-optimal if and
only if
\begin{equation}\label{eq6}
v_{\alpha}(x)=c(x,\phi_{\alpha}(x))+\alpha \int_\X
v_{\alpha}(y)q(dy|x, \phi_{\alpha}(x)),\quad x\in \X.
\end{equation}

\section{General Assumptions and Auxiliary Results}\label{s3}
Following Sch\"al \cite{Schal}, consider the following assumption.

\textbf{Assumption (${\rm \bf G}$)}. $ w^*:=\inf\limits_{x\in
\X}w(x)<+\infty$.

This assumption is equivalent to the existence of $x\in\X$ and
$\pi\in\Pi$ with $w^\pi(x)<\infty.$  If Assumption~(${\rm \bf G}$)
does not hold then the problem is trivial, because $w(x)=\infty$
for all $x\in\X$ and any policy $\pi$ is average-cost optimal.
Define the following quantities for $\alpha\in [0,1)$:
\begin{equation*}
m_{\alpha}=\inf\limits_{x\in \X}v_{\alpha}(x),\quad
u_{\alpha}(x)=v_{\alpha}(x)-m_{\alpha},
\end{equation*}
\begin{equation*}
\underline{w}=\ilim\limits_{\alpha\uparrow
1}(1-\alpha)m_{\alpha},\quad\overline{w}=\slim\limits_{\alpha\uparrow
1}(1-\alpha)m_{\alpha}.
\end{equation*}
Observe that $u_\alpha(x)\ge 0$ for all $x\in\X.$ According to
Sch\"al \cite[Lemma 1.2]{Schal}, Assumption~(${\rm \bf G}$) implies
\begin{equation}\label{eq:schal} 0\le \underline{w}\le \overline{w}\le w^*<
+\infty.
\end{equation}

According to Sch\"al~\cite[Proposition 1.3]{Schal}, under
Assumption~(${\rm \bf G}$),  if there exists a measurable function
$u:\X\to [0,+\infty)$ and a stationary policy $\phi$ such that
\begin{equation}\label{eq7}
\underline{w}+u(x)\ge c(x, \phi(x))+\int_\X  u(y)q(dy|x,
\phi(x)),\quad x\in \X,
\end{equation}
then $\phi$ is \textit{average-cost optimal} and
$w(x)=w^*=\underline{w}=\overline{w}$ for all $x\in \X.$  Here
need a different form of such a statement.

\begin{theorem}\label{Prop1}
Let Assumption~(${\rm \bf G}$) hold. If there exists a measurable
function $u:\X\to [0,+\infty)$ and a stationary policy $\phi$ such
that
\begin{equation}\label{eq7111}
\overline{w}+u(x)\ge c(x, \phi(x))+\int_\X  u(y)q(dy|x,
\phi(x)),\quad x\in \X,
\end{equation}
then $\phi$ is average-cost optimal and
\begin{equation}\label{eq:7121}
w(x)=w^{\phi}(x)=\slim\limits_{\alpha\uparrow
1}(1-\alpha)v_{\alpha}(x)=\overline{w}=w^*,\quad x\in \X.
\end{equation}
\end{theorem}
\begin{proof}
Similarly to Hern\'andez-Lerma \cite[p.~239]{HLerma} or Sch\"al
\cite[Proposition~1.3]{Schal}, since $u$ is nonnegative, by
iterating (\ref{eq7111}) we obtain
\[
n\overline{w}+u(x)\ge v_n^\phi(x), \quad n\ge 1,\ x\in \X.
\]
Therefore, after dividing the last inequality by $n$ and setting
$n\to\infty$, we have
\begin{equation}\label{eq***}
\overline{w}\ge w^{\phi}(x)\ge w(x)\ge w^*,\quad  x\in \X,
\end{equation}
where the second and the third inequalities follow from the
definitions of $w$ and $w^*$ respectively. Since $\overline{w}\ge
w^*$, inequalities (\ref{eq:schal}) imply that for all $\pi\in\Pi$
\[
w^*=\overline{w} \le \slim\limits_{\alpha\uparrow
1}(1-\alpha)v_{\alpha}(x)\le\slim\limits_{\alpha\uparrow
1}(1-\alpha)v_{\alpha}^{\pi}(x)\le w^{\pi}(x), \qquad \pi\in \Pi,
\ x\in\X.
\]
Finally, we obtain that
\begin{equation}\label{eq3.5}
w^*=\overline{w}\le \slim\limits_{\alpha\uparrow
1}(1-\alpha)v_{\alpha}(x)\le \inf\limits_{\pi\in
\Pi}w^{\pi}(x)=w(x)\le w^{\phi}(x) \le\overline{w},\qquad x\in\X,
\end{equation}
where the last inequality 
follows from (\ref{eq***}).  Thus all the inequalities in
(\ref{eq3.5}) are equalities.
\end{proof}

Let us set $\mathbb{R}=[-\infty,+\infty)$,
$\mathbb{R}_+=[0,\infty)$, and
$\overline{\mathbb{R}}=\mathbb{R}\cup\{+\infty\}.$ For an
$\overline{\mathbb{R}}$-valued function $f$, defined on a Borel
subset $U$ of a Polish space $\mathbb{Y},$ consider the level sets
\begin{equation}\label{def-D}\mathcal{D}_f(\lambda)=\{y\in U \, : \,  f(y)\le
\lambda\},\end{equation} $-\infty<\lambda< +\infty.$ We recall
that the function $f$ is \textit{lower semi-continuous on $U$} if
all the level sets $\mathcal{D}_f(\lambda)$ are closed and the
function is \textit{inf-compact on $U$} if all these sets are
compact. The level sets $\mathcal{D}_f(\lambda)$ satisfy the
following properties that are used in this paper:

(a) if $\lambda_1>\lambda$ then $\mathcal{D}_f(\lambda)\subseteq
\mathcal{D}_f(\lambda_1);$

(b) if $g,f$ are functions on $U$ satisfying $g(y)\ge f(y)$ for all
$y\in U$ then $\mathcal{D}_g(\lambda)\subseteq
\mathcal{D}_f(\lambda).$

A set is called $\sigma$-compact if it is a union of a countable
number of compact sets.  Denote by $K(\mathbb{A})$ the
\textit{family of all nonempty compact subsets} of $\mathbb{A}$
and by $K_\sigma(\mathbb{A})$ \textit{family of all 
$\sigma$-compact subsets} of $\mathbb{A}$; $K(\mathbb{A})\subset
K_\sigma(\mathbb{A})$. Also denote by $S(\A)$ the set of nonempty
subsets of $\A.$

A set-valued mapping ${F}:\X \to S(\mathbb{A})$ is \textit{upper
semi-continuous} at $x\in\X$ if, for any neighborhood $G$ of the set
$F(x)$, there is a neighborhood of $x$, say $U(x)$, such that
$F(y)\subseteq G$ for all $y\in U(x)$ (see e.g., Berge
\cite[p.~109]{Ber} or Zgurovsky et al. \cite[Chapter~1,
p.~7]{ZMK1}). A set-valued mapping is called \textit{upper
semi-continuous}, if it is upper semi-continuous at all $x\in\X$.

For weakly continuous transition probabilities, the following
basic assumptions were considered in Sch\"al~\cite{Schal}.

\textbf{Assumption (${\rm \bf W}$)}.

(i) $c$ is lower semi-continuous and bounded below on ${\rm Gr}
({A})$;

(ii) $A(x)\in K(\mathbb{A})$ for $x\in\X$ and $A:\X\to
K(\mathbb{A})$ is upper semi-continuous;

(iii) the transition probability  $q(\cdot|x,a)$ is weakly
continuous in $(x,a)\in {\rm Gr}(A).$

\textit{Weak continuity} of $q$ in $(x,a)$ means that
$$
\int_\X  f(z)q(dz|x_k,a_k)\to \int_\X  f(z)q(dz|x,a),\qquad\qquad
k=1,2,\ldots,
$$
for any sequence $\{(x_k,a_k),k\ge 0\}$ converging to $(x,a),$
where $(x_k,a_k),$ $(x,a)\in {\rm Gr} ({A}),$ and for any bounded
continuous function $f: \X\to \mathbb{R}$. We notice that there is
an additional assumption in Sch\"al~\cite{Schal}, namely,  that
$\X$ is a locally compact space with countable base. However, as
follows from this paper, the assumption is not necessary here as
well as  in Feinberg and Lewis \cite{FL}, since there exists at
least one stationary policy. We also remark that the assumptions
in (${\rm \bf W}$) were presented in a different order here than
in Sch\"al~\cite{Schal}, and that it is assumed in
Sch\"al~\cite{Schal} that $c$ is nonnegative.  Since for
discounted and average cost criteria the cost function can be
shifted by adding any constant, the boundness and nonnegativity of
$c$ are equivalent assumptions. We consider Assumption~(${\rm \bf
Wu}$) from Feinberg and Lewis \cite{FL} without assuming that $\X$
is locally compact.

\textbf{Assumption (${\rm \bf Wu}$)}.

(i) $c$ is inf-compact on ${\rm Gr} ({A})$;

(ii) Assumption (${\rm \bf W}$)(iii) holds.

%

\textbf{Assumption (${\rm \bf W^*}$)}.

(i) Assumption~(${\rm \bf W}$)(i) holds;


(ii) if a sequence $\{x_n \}_{n=1,2,\ldots}$ with values in $\X$
converges and its limit $x$ belongs to $\X$ then any sequence
$\{a_n \}_{n=1,2,\ldots}$ with $a_n\in A(x_n)$, $n=1,2,\ldots,$
satisfying the condition that the sequence $\{c(x_n,a_n)
\}_{n=1,2,\ldots}$ is bounded above, has a limit point $a\in
A(x);$

(iii) Assumption (${\rm \bf W}$)(iii) holds.

%

\begin{lemma}\label{lm0}
The following statements hold:

(i) Assumption (${\rm \bf W}$) implies Assumption (${\rm \bf W^*}$);

(ii) Assumption (${\rm \bf Wu}$) implies Assumption (${\rm \bf
W^*}$).
\end{lemma}
\begin{proof} (i) Let $x_n\to x$ as $n\to\infty$, where
$x\in \X$ and $x_n\in\X$, $n=1,\ldots\ .$  We show that under
Assumption~(${\rm \bf W}$)(ii) any sequence $\{a_n
\}_{n=1,2,\ldots}$ with $a_n\in A(x_n)$ has a limit point $a\in
A(x)$. Indeed, since $\mathcal{K}:=\left(\cup_{n\ge 1}
\{x_n\}\right)\cup\{x\}$ is a compact set and set-valued mapping
$A:\X\to K(\A)$ is upper semi-continuous, then
Berge~\cite[Theorem~3 on p. 110]{Ber} implies that the image
$A(\mathcal{K})$ is also compact. As $\{a_n\}_{n\ge 1}\subset
A(\mathcal{K})$ then the sequence $\{a_n\}_{n\ge 1}$ has a limit
point $a\in \A.$ Consider a sequence $n_k\to \infty$ such that
$a_{n_k}\to a.$   Since $A(z)\in K(\A)$ for all $z\in X$, the
upper-semicontinuous set-valued mapping $A$ is closed and, since
$A$ is closed, $a\in A(x);$ Berge~\cite[Theorems 5 and 6 on pp.
111, 112]{Ber}.


(ii)   Since  $c$ is inf-compact, it is lower-semicontinuous and
bounded below.  We just need to show that  Assumption (${\rm \bf
W^*}$)(ii) holds.  Let us consider $x_n\to x$ as $n\to+\infty$ and
$a_n\in A(x_n)$, $n= 1,,2,\ldots,$ such that $x_n,x\in\X$ and for
some $\lambda<\infty$ the inequality $ c(x_n,a_n)\le \lambda$ holds
for all $n=1,2,\ldots\ .$ Then, by inf-compactness of $c$ on ${\rm
Gr}(A)$, the level set $\mathcal{D}_c(\lambda)$ is compact. Thus the
sequence $\{x_n,a_n\}_{n\ge 1}$ has a limit point $(x,a)\in
\mathcal{D}_c(\lambda)\subseteq{\rm Gr}(A)$. Since $(x,a)\in {\rm
Gr}(A)$, we have $a\in A(x).$
\end{proof}

For any $\alpha\ge 0$ and lower semi-continuous nonnegative
function $u:\X\to \overline{\mathbb{R}}$, we consider an
\textit{operation} $\eta_u^\alpha$,
\begin{equation}\label{e:defeta}
\eta_u^\alpha(x,a)=c(x,a)+\alpha\int_\X  u(y)q(dy| x,a),
\quad(x,a)\in {\rm Gr}(A).
\end{equation}

Let $L(\mathbb{X})$ be the class of all lower semi-continuous and
bounded below functions $\varphi:\mathbb{X}\to\overline{\mathbb{R}}$
with ${\rm dom\,}\varphi:=\{x\in\X\, : \, \varphi(x)<+\infty\}\ne
\emptyset$. Observe that $\eta_u^\alpha=\eta_{\alpha u}^1  $.

\begin{lemma}\label{lm00101} For any $x\in\X$ the following
statements hold:

(a) under Assumption ${\rm \bf W^*}$(ii), the function
$c(x,\cdot)$ is inf-compact on $A(x)$;

(b) under Assumptions ${\rm \bf W^*}$(ii,iii), for any $u\in
L(\mathbb{X})$ and $\alpha\ge 0$, the function
$\eta_u^\alpha(x,\cdot)$ is inf-compact on $A(x)$.
\end{lemma}
\begin{proof}
(a) For an arbitrary $\lambda\in\R$ and fixed $x\in \X$, consider
the set $\mathcal{D}_{c(x,\cdot)}(\lambda)=\{a\in A(x)\,:\,
c(x,a)\le\lambda\}.$  Assumption ${\rm \bf W^*}$(ii) means, that
this set is compact.  Thus, (i) is proved.

(b) Fix $x\in\X$ again.  Since $u\in L(\X)$ and $q$ is weakly
continuous in $a$, the second summand in (\ref{e:defeta}) is a lower
semi-continuous function on $A(x)$ (Hern\'ndez-Lerma and
Lasserre~\cite[p. 185]{HLerma1}) and it is bounded below by the same
constant as $u.$ According to statement (i), $c(x,\cdot)$ is
inf-compact on $A(x)$. The sum of an inf-compact function and a
bounded below lower semi-continuous function is an inf-continuous
function.
\end{proof}

A measurable mapping $\phi:\, \X\to\A$, such that $\phi(x)\in A(x)$
for all $x\in \X$, is called a selector (or a measurable selector).
In our case, selectors and decision rules are the same objects.
Since we identify a stationary policy with a decision rule,
selectors and stationary policies are the same objects. The
existence of selector for the mapping $A$ is the necessary and
sufficient condition for the existence of a policy.   Let
$E\subseteq \X\times\A$ and ${\rm proj}_{\X}\ E=\{x\in\X\,:\,
(x,a)\in E {\rm \ for\ some\ }a\in E\}$ be a projection  of $E$ on
$X$.  A Borel map $f:\ {\rm proj}_{\X}\ E \to \A$ is called a Borel
uniformization of $E$, if $(x,f(x))\in E$ for all $x\in {\rm
proj}_{\X}\ E$. Let $E_x=\{a\,:\, (x,a)\in E\}$ be a cut of $E$ at
$x\in\X.$

\vskip 0.2 cm

\noindent \textbf{Arsenin-Kunugui Theorem}  (Kechris~\cite[p.
297]{Kec}) \textit{ If $E$ is a Borel subset of  $\X\times\A$ and
$E_x\in K_\sigma(\A)$ for all $x\in \X$ then there exists a Borel
uniformization of $E$ and ${\rm proj}_{\X}\ E$ is a Borel set.}

\vskip 0.2 cm

We remark that it is assumed in Kechris~\cite[p. 297]{Kec}) that
$\X$ is a standard Borel space (that is, isomorphic to a Borel
subset of a Polish space) and $\A$ is a Polish space.  Here $\X$
and $\A$ are Borel subsets of Polish spaces.  These two
formulations are obviously equivalent.

We recall that ${\rm Gr}(A)$ is assumed to be Borel and
$A(x)\ne\emptyset,$ $x\in\X.$ With $E={\rm Gr}(A)$, Arsenin-Kunugui
Theorem implies the existence of a stationary policy under the
assumption $A(x)\in K(\A),$ $x\in\X.$  Thus, Assumption~(${\rm \bf
W}$) implies the existence of a policy for the MDP.

Let  Assumption~(${\rm \bf W^*}$) hold.  Set $F(x)=\{a\in
A(x)\,:\, c(x,a)<\infty\},$ $x\in\X.$ In view of
Lemma~\ref{lm00101},
$F(x)=\cup_{n\in\{1,2,\ldots\}}\mathcal{D}_{c(x,\cdot)}(n)\in
K_\sigma(\A).$  In addition, ${\rm Gr}(F)=\{(x,a)\in {\rm
Gr}(A)\,:\, c(x,a)<\infty\}$ is a Borel subset of $\X\times\A.$
Thus, if the function $c$ takes only finite values, a stationary
policy exists in view of Arsenin-Kunugui Theorem.

Of course, if it is possible that $c(x,a)=\infty$, a
uniformization may not exist. For example, this takes place when
$c(x,a)=\infty$ for all $(x,a)\in {\rm Gr}(A)$ and ${\rm Gr}(A)$
does not have a measurable selector. However $c(x,a)=\infty$ means
from a modeling prospective that this state-action pair should be
excluded, because selecting $a$ in $x$ leads to the worst possible
result.  If there are state-action pairs $(x,a)$ with
$c(x,a)=\infty$ and ${\rm Gr}(A)$ does not have a uniformization,
the MDP can be transformed into an MDP modeling the same problem
and with a nonempty set of policies. Let us exclude the situation
when $c(x,a)=\infty$ for all $(x,a)\in {\rm Gr}(A)$, because it is
trivial: all the actions are bad. Define $X={\rm proj}_{\X}\ {\rm
Gr}(F)$ and $Y=\X\setminus X.$ Under Assumption~(${\rm \bf W}^*$),
Arsenin-Kunigui Theorem implies that $X$ is Borel and there exist
a Borel mapping $f$ from $X$ to $\A$ such that $f(x)\in F(x)$ for
all $x\in X.$ If $Y=\emptyset$ (that is, there exists an action
$a\in A(x)$ with $c(x,a)<\infty$ for each $x\in\X$) then $\phi=f$
is a stationary policy.

Let us consider the situation when $Y\ne\emptyset.$  In such an
MDP, as soon as the state is in $Y$, the losses are infinite and
there is no reason to model the process after this.  Let us
transform the model by  choosing any $x^*\in Y$ and any $a^*\in\A$
and setting the new state set $\X^*=X\cup \{x^*\}$, keeping the
original action set $\A$, setting new action sets $A^*(x)=F(x)$
for $x\in X$ and $A^*(x^*)=\{a^* \},$ defining the new cost
function
\begin{equation*}
c^*(x,a)=\begin{cases} c(x,a), &{\rm if\ }
x\in Y  {\rm \ and}\  a\in F(x), \\
\infty, &{\rm if\ } x=x^*\ {\rm and}\ a=a^*.
\end{cases}
\end{equation*}
and considering new transition probabilities defined for $x\in
X^*$ and $a\in A^*(x)$ by
\begin{equation*}
q^*(B|x,a)=\begin{cases} q(B|x,a), &{\rm if\ }
B\subseteq X,\ B\in {\cal B}(\X),\ {\rm and}\ \ x\in X,\\
q(Y|x,a), &{\rm if\ }  B=\{x^*\},\  {\rm and}\ x\in X,\\
1, &{\rm if\ } B=\{x^*\}\ {\rm and}\ x=x^*.
\end{cases}
\end{equation*}
The new MDP is nontrivial in the sense that the set of policies is
not empty.  Finding an optimal policy for this MDP is equivalent to
finding a policy for the original MDP until its first exit time from
$X$, and in both cases the process incurs infinite losses, if it
leaves $X$.   So, the original and the new MDP model are the same
problem.

\begin{lemma}\label{lm1}
If  Assumption~(${\rm \bf W^*}$) holds and $u\in L(\mathbb{X})$,
then the function
\begin{equation}\label{eq:lm10}
u^*(x):=\inf\limits_{a\in A(x)}\big[c(x,a)+\int_\X
u(y)q(dy|x,a)\big],\qquad x\in \mathbb{X},
\end{equation}
belongs to $L(\mathbb{X}),$ and there exists $f\in \mathbb{F}$
such that
\begin{equation}\label{eq:lm00}
u^*(x)=c(x,f(x))+\int_\X  u(y)q(dy|x,f(x)),\qquad x\in \mathbb{X}.
\end{equation}
Moreover, infimum in (\ref{eq:lm10}) can be replaced by minimum, and
the  nonempty sets
\begin{equation}\label{eq:lm000}
A_*(x)=\left\{a\in A(x):\,u^*(x)=c(x,a)+ \int_\X
u(y)q(dy|x,a)\right\},\qquad x\in \X,
\end{equation}
satisfy the following properties:

(a) the graph ${\rm Gr}(A_*)=\{(x,a):\, x\in\X, a\in A_*(x)\}$ is
a Borel subset of $\X\times \mathbb{A}$;

 (b) if $u^*(x)=+\infty$,
then $A_*(x)=A(x)$, and, if $u^*(x)<+\infty$, then $A_*(x)$ is
compact.
\end{lemma}
\begin{proof}
Under Assumption~(${\rm \bf W^*}$), for any lower semi-continuous on
$\X$, bounded below function $u:\X\to \overline{\mathbb{R}}$ and
$\alpha\in(0,1]$, the function $\eta_{u(x,\cdot)}^\alpha$ is
inf-compact on $A(x)$, $x\in\X$. This follows from
Lemma~\ref{lm00101}. Thus, infimum in (\ref{eq:lm10}) can be
replaced by minimum and $A^*(x)$ is nonempty for any $x\in \X$.

Now we show that $u^*$ is lower semi-continuous on $\X$. Let us
fix an arbitrary $x\in\X$ and any sequence $x_n\to x$ as
$n\to+\infty$. We need to prove the inequality
\begin{equation}\label{eq:lm11}
u^*(x)\le\ilim\limits_{n\to+\infty}u^*(x_n).
\end{equation}
If $\ilim\limits_{n\to+\infty}u^*(x_n)=+\infty$, then
(\ref{eq:lm11}) obviously holds. Thus we consider the case, when
$\ilim\limits_{n\to+\infty}u^*(x_n)<+\infty$. There exists a
subsequence $\{x_{n_k}\}_{k\ge 1}\subseteq\{x_n\}_{n\ge 1}$ such
that
\[
\ilim\limits_{n\to+\infty}u^*(x_n)=\lim\limits_{k\to+\infty}u^*(x_{n_k}).
\]
Setting $\lambda=\lim\limits_{k\to+\infty}u^*(x_{n_k})+1$, we get
the inequality $u^*(x_{n_k})\le\lambda$ for all  $ k\ge K$, where
$K$ is some natural number. Since the function $\eta_u^1$ is
inf-compact on ${\rm Gr}(A)$, equation (\ref{eq:lm10}) can be
rewritten as
\[
u^*(x):=\min\limits_{a\in A(x)}\eta_u^1(x,a),\,\, x\in \mathbb{X}.
\]
Thus, for any $k\ge K$ there exists $a_k\in A(x_{n_k})$ such that
$u^*(x_{n_k})=\eta_u^1(x_{n_k},a_k)$. Therefore,
\[
 c(x_{n_k},a_k)\le \eta_u^1(x_{n_k},a_k)\le \lambda, \qquad k\ge K. 
\]
In  view of Assumption (${\rm \bf W^*}$)(ii), there exists a
convergent subsequence $\{a_{k_m}\}_{m\ge 1}$ of the sequence
$\{a_k\}_{k\ge 1}$ such that $a_{k_m}\to a\in A(x)$ as
$m\to+\infty$. Due to lower semi-continuity of $\eta_u^1$ on ${\rm
Gr}(A)$, \[
\ilim\limits_{n\to+\infty}u^*(x_n)=\lim\limits_{k\to+\infty}u^*(x_{n_k})=
\lim\limits_{m\to+\infty}u^*(x_{n_{k_m}})=\lim\limits_{m\to+\infty}\eta_u^1(x_{n_{k_m}},a_{k_m})
\ge\eta_u^1(x,a)\ge u^*(x).
\]
Inequality (\ref{eq:lm11}) holds. Thus, $u^*$ is lower
semi-continuous on $\X$.

Now we consider the nonempty sets $A_*(x)$, $x\in\X$, defined in
(\ref{eq:lm000}). The graph ${\rm Gr}(A_*)$ is a Borel subset of
$\X\times \mathbb{A}$, because ${\rm Gr}(A_*)=\{(x,a)\,:
u^*(x)=\eta_{u}^1(x,a) \}$, and the functions $\eta_{u}^1$ and $u^*$
are lower semi-continuous on ${\rm Gr}(A)$ and $\X$ respectively,
and therefore they are Borel.

 We remark that, if $u^*=+\infty$, then
$A_*(x)=A(x)$. If $u^*(x)<\infty$, then Lemma~\ref{lm00101} implies
that the set $A_*(x)$ is compact. Indeed, fix any
$x\in\X_f:=\{x\in\X\,:\, u^*(x)<\infty\}$ and set $\lambda=u^*(x)$.
Then the set $ A_*(x)=\{a\in A(x)\,: \, \eta_{u}^1(x,a)\le\lambda
\}=\mathcal{D}_{\eta_{u}^1(x,\cdot)}(\lambda)$   is compact,
because $\eta_{u}^1(x,\cdot)$ is inf-compact on $A(x)$. 

Let us prove the existence of $f\in\mathbb{F}$ satisfying
(\ref{eq:lm00}). Since the function $u^*$ is lower-semicontinuous,
it is Borel and the sets $X_\infty:=\{x\in\X\,:\, u^*(x)=+\infty\}$
and $\X_f$ are  Borel. Therefore, the graph of the mapping $\X_f\to
A_*$ is the Borel set ${\rm Gr}(A_*)\setminus (\X_\infty\times\A)$.
Since the nonempty  sets $A_*(x)$ are compact for all $x\in\X_f$,
the Arsenin-Kunugui Theorem implies the existence of a Borel
selector $f_1:\, \X_f\to \A$ such that $f_1(x)\in A_*(x)$ for all
$x\in\X.$ Consider any Borel mapping $f_2$ from $\X$ to $\A$
satisfying $f_2(x)\in A(x)$ for all $x\in\X$ and set
\[
f(x)=\begin{cases} f_1(x), &{\rm if\ } x\in\X_f,\\
f_2(x), &{\rm if\ } x\in\X_\infty.
\end{cases}
\]
Then $f\in\F$ and $f(x)\in A_*(x)$ for all $x\in\X.$
%
%
\end{proof}

The following Lemma~\ref{lemma2} is formulated in Sch\"al
\cite[Lemma 2.3(ii)]{Schal} without proof. Reference Serfozo
\cite{Serfozo} mentioned in Sch\"al \cite[Lemma 2.3(ii)]{Schal}
contains relevant facts, but it does not contain this statement.
Therefore we provide the proof.
Recall that for a metric space $S$,  the family of all probability
measures on $(S, {\cal B}(S))$ is denoted by $\mathbb{P}(S)$.

\begin{lemma}\label{lemma2}
Let $S$ be an arbitrary metric space, $\{\mu_n\}_{n\ge 1} \subset
\mathbb{P}(S)$ converges  weakly to $\mu\in \mathbb{P}(S)$, and
$\{h_n\}_{n\ge 1}$ be a sequence of measurable nonnegative
$\overline{\mathbb{R}}$-valued functions on $S$. Then
\[
\int_S \underline{h}(s)\mu(ds)\le \ilim\limits_{n\to
+\infty}\int_S h_n(s)\mu_n(ds),
\]
where $\underline{h}(s)=\ilim\limits_{n\to+\infty,\, s'\to
s}h_n(s')$, $s\in S$.
\end{lemma}
\begin{proof}  See Appendix {\bf A}. \end{proof}

We remark that $\ilim\limits_{n\to+\infty,\, s'\to s}h_n(s')$ is
the least upper bound of the set of all $\lambda\in \R$ such that
there exist $N=1,2,\ldots$ and a neighborhood $U(s)$ of $s$ such
that $\lambda\le \inf\{h_n(s'):\, n\ge N, s'\in U(s)\}.$

\section{Expected Total Discounted  Costs}\label{s4}

In this section, we establish under Assumption~(${\rm \bf W^*}$) the
standard properties of discounted MDPs: the existence of stationary
optimal policies, description of the sets of stationary optimal
policy, and convergence of value iterations.
Theorem~\ref{prop:dcoe} strengthens Feinberg and
Lewis~\cite[Proposition 3.1]{FL}, where these facts are proved under
Assumption~(${\rm \bf Wu}$).  In terms of applications to inventory
and queuing control, Assumption~(${\rm \bf W^*}$) does not require
that holding costs increase to infinity as the  inventory level (or
workload, or the number of customers in queue) increases to
infinity.

\begin{theorem}\label{prop:dcoe} Let Assumption~(${\rm \bf W^*}$) hold.
Then

{(i}) the functions $v_{n,\alpha}$, $n=1,2,\ldots$, and $v_\alpha$
are lower semi-continuous on $\X$, and $v_{n,\alpha}(x)\uparrow
v_\alpha (x)$ as $n \to +\infty$ for all $x\in \X;$

{(ii)} 
\begin{equation}\label{eq433}
v_{n+1,\alpha}(x)=\min\limits_{a\in A(x)}\left\{c(x,a)+\alpha
\int_\X v_{n,\alpha}(y)q(dy|x,a)\right\},\quad x\in
\mathbb{X},\,\,n=0,1,...,
\end{equation}
where $v_{0,\alpha}(x)=0$ for all $x\in \X$, and the nonempty sets
$A_{n,\alpha}(x):=\{a\in
A(x):\,v_{n+1,\alpha}(x)=\eta_{v_{n,\alpha}}^\alpha(x,a) \}$, $x\in
\X$, $n=0,1,\ldots,$ satisfy the following properties: (a) the graph
${\rm Gr}(A_{n,\alpha})=\{(x,a):\, x\in\X, a\in A_\alpha(x)\}$,
$n=0,1,\ldots,$ is a Borel subset of $\X\times \mathbb{A}$, and (b)
if $v_{n+1,\alpha}(x)=+\infty$, then $A_{n,\alpha}(x)=A(x)$ and, if
$v_{n+1,\alpha}(x)<+\infty$, then $A_{n,\alpha}(x)$ is compact;

{(iii)} for any $N=1,2,\ldots$, there exists a Markov optimal
$N$-horizon policy $(\phi_0,\ldots,\phi_{N-1})$ and if, for an
$N$-horizon Markov policy $(\phi_0,\ldots,\phi_{N-1})$ the
inclusions $\phi_{N-1-n}(x)\in A_{\alpha,n}(x)$, $x\in\X,$
$n=0,\ldots,N-1,$ hold then this policy is $N$-horizon optimal;

{(iv)} for $\alpha\in [0,1)$
\begin{equation}\label{eq5a}
v_{\alpha}(x)=\min\limits_{a\in A(x)}\left\{c(x,a)+\alpha\int_\X
v_{\alpha}(y)q(dy|x,a)\right\},\qquad x\in \X,
\end{equation}
and the nonempty sets $A_{\alpha}(x):=\{a\in
A(x):\,v_{\alpha}(x)=\eta_{v_{\alpha}}^\alpha(x,a) \}$, $x\in \X$,
satisfy the following properties: (a) the graph ${\rm
Gr}(A_{\alpha})=\{(x,a):\, x\in\X, a\in A_\alpha(x)\}$  is a Borel
subset of $\X\times \mathbb{A}$, and (b) if $v_{\alpha}(x)=+\infty$,
then $A_{\alpha}(x)=A(x)$ and, if $v_{\alpha}(x)<+\infty$, then
$A_{\alpha}(x)$ is compact.

{(v)} for an infinite-horizon there exists a stationary
discount-optimal policy $\phi_\alpha$, and a stationary policy is
optimal if and only if $\phi_\alpha(x)\in A_\alpha(x)$ for all
$x\in \X.$

{(vi)} {\rm (Feinberg and Lewis~\cite[Proposition 3.1(iv)]{FL})} under Assumption~(${\rm \bf Wu}$),  the functions
$v_{n,\alpha}$, $n=1,2,\ldots$, and $v_\alpha$ are inf-compact on
$\X$.
%
%
\end{theorem}
\begin{proof} 
%
(i)--(v). First, we prove these statements for a nonnegative cost function $c$.
In this case, $v_{n,\alpha}(x)\ge 0$, $n=0,1,\ldots,$ and $v_\alpha(x)\ge 0$ for all $x\in\X.$

By (\ref{eq3_3}) and Lemma~\ref{lm1}, $v_{1,\alpha}\in L(\X)$, since
$v_{0,\alpha}=0\in L(\X).$ By the same arguments, if
$v_{n,\alpha}\in L(\X)$ then $v_{n+1,\alpha}\in L(\X)$.  Thus
$v_{n,\alpha}\in L(\X)$ for all $n=0,1,\ldots\ .$ By
Lemma~\ref{lm00101}, for any $n=1,2,\ldots$, $x\in\X$, and
$\lambda\in\mathbb{R},$ the set
$\mathcal{D}_{\eta_{v_{n,\alpha}}^\alpha(x,\cdot)}(\lambda)$ is a
compact subset of $\mathbb{A}$. By Bertsekas and
Shreve~\cite[Proposition~9.17]{Bert}, $v_{n,\alpha}\uparrow
v_{\alpha}$ as $n\to+\infty$.
Since 
the limit of a monotone increasing sequence of lower semi-continuous
functions is again a lower semi-continuous function, $v_\alpha\in
L(\X).$  Lemma~\ref{lm1}, applied to equations (\ref{eq3_3}) and
(\ref{eq5}),   implies statements (ii) and (iv)
 respectively.  Statement (iii) follows from (\ref{eq4_4}) and statement (v) follows from (\ref{eq6}).

Now let $c(x,a)\ge K$ for all $(x,a)\in {\rm Gr}(A)$ and for some
$K>-\infty.$  For $K\ge 0$, statements (i)--(v) are proved.  For
$K<0$, consider the value functions ${\tilde c}=c-K\ge 0$. If the
cost function $c$ substituted with $\tilde c$, we substitute the
notation $v$ with $\tilde v.$ Then $v_{n,\alpha}^\pi={\tilde
v}_{n,\alpha}^\pi + \frac{1-\alpha^n}{1-\alpha}K,$ $n=0,1,\ldots,$
for all policies $\pi.$ Thus, $v_{n,\alpha}={\tilde v}_{n,\alpha} +
\frac{1-\alpha^n}{1-\alpha}K,$ $n=0,1,\ldots,$ and
$v_{\alpha}={\tilde v}_{\alpha} + \frac{K}{1-\alpha}.$ Since
statements (i)--(v) hold for the shifted costs $\tilde c$ and the
value functions ${\tilde v}_{n,\alpha}$ and ${\tilde v}_\alpha$,
they also hold for the initial cost function $c$ and the value
functions $v_{n,\alpha}$ and $v_\alpha.$ \end{proof}

We remark that the conclusions of Theorem~\ref{prop:dcoe} and its
proof remain correct when $\alpha=1$ and the function $c$ is
nonnegative.

\section{Average Costs Per Unit Time}\label{s5} $ $

In this section we show that Assumption~(${\rm \bf W^*}$) and
boudness assumption Assumption (${\rm \bf \underline{B}}$) on the
function $u_\alpha$, which is weaker boundness Assumption~(${\rm \bf
B}$) introduced by Sch\"al~\cite{Schal}, lead to the validity of
stationary average-cost optimal inequalities and the existence of
stationary policies. Stronger results hold under Assumption~(${\rm
\bf B}$).

\textbf{Assumption (${\rm \bf \underline{B}}$).} (i)
Assumption~(${\rm \bf G}$)
 holds, and (ii) $\ilim\limits_{\alpha \uparrow
1}u_{\alpha}(x)<\infty$ for all $x\in \X$.

Assumption~(${\rm \bf \underline{B}}$)(ii) is weaker than the
assumption
$\sup_{\alpha\in [0,1)}u_{\alpha}(x)<\infty$ for all $x\in \X$
considered in Sch\"al~\cite{Schal}.  This assumption and
Assumption~(${\rm \bf G}$) were combined in Feinberg and
Lewis~\cite{FL} into the following assumption.

\textbf{Assumption (${\rm \bf B}$).} (i) Assumption~(${\rm \bf
G}$) holds, and (ii) $\sup_{\alpha\in [0,1)}u_{\alpha}(x)<\infty$
for all $x\in \X$.

It seems natural to consider the assumption $\slim\limits_{\alpha
\uparrow 1}u_{\alpha}(x)<\infty$ for all $x\in \X$, which is
stronger than Assumption~(${\rm \bf \underline{B}}$)(ii) and
weaker than Assumption (${\rm \bf B}$)(ii).  However, as the
following lemma shows, under Assumption~(${\rm \bf G}$) this
assumption is equivalent to Assumption (${\rm \bf B}$)(ii).

\begin{lemma}\label{l5.1EF}
Let the cost function $c$ be bounded below and Assumption~(${\rm
\bf G}$) hold.  Then for each $x\in\X$ the following two
inequalities are equivalent:

(i) $\sup_{\alpha\in [0,1)}u_{\alpha}(x)<\infty$,

(ii) $\slim\limits_{\alpha \uparrow 1}u_{\alpha}(x)<\infty$.
\end{lemma}
\begin{proof} Obviously, (i)$\to$(ii).  Let us prove (ii)$\to$(i). Let (ii) hold.  Assume that (i) does not hold.
Since  $\sup_{\alpha\in [0,1)}u_{\alpha}(x)= \max\{\sup_{\alpha\in
[0,\alpha^*)}u_{\alpha}(x), \sup_{\alpha\in
[\alpha^*,1)}u_{\alpha}(x)\}$ for any $\alpha^*\in [0,1)$, there
exists $\alpha^*\in [0,1)$ such that $\sup_{\alpha\in
[0,\alpha^*)}u_{\alpha}(x)=\infty.$

 Since the function $u_\alpha$
remains unchanged, if a finite constant is added to the cost
function $c$, we assume without loss of generality that $c(x,a)\ge
0$ for all $(x,a)\in {\rm Gr}(A).$  Since $c\ge 0$, the functions
$v_\alpha(x)$ and $m_\alpha$ are nonnegative nondecreasing
functions in $\alpha\in [0,1).$  Since
$v_\alpha(x)=u_\alpha(x)+m_\alpha\ge u_\alpha(x)$, we have
$\sup_{\alpha\in [0,\alpha^*)}v_{\alpha}(x)=\infty$ and therefore
$v_\alpha(x)=\infty$ for all $\alpha\in [\alpha^*,1),$ because of
the monotonicity of $v_\alpha$ in $\alpha$.  Thus,
$\slim\limits_{\alpha \uparrow 1}(1-\alpha)v_{\alpha}(x)=\infty.$
However, $\slim\limits_{\alpha \uparrow
1}(1-\alpha)v_{\alpha}(x)=\slim\limits_{\alpha \uparrow
1}(1-\alpha)(u_{\alpha}(x)+m_\alpha)\le \slim\limits_{\alpha
\uparrow 1}(1-\alpha)u_{\alpha}(x)+ \overline{w}<\infty,$ where
the last inequality follows from (ii) and (\ref{eq:schal}).
 The obtained contradiction completes the proof.
\end{proof}

Until the end of this section we assume that Assumption~(${\rm \bf
\underline{B}}$) holds. Let us set
\begin{equation}\label{eq131}
u(x):=\ilim\limits_{\small \alpha\uparrow 1, \ y\to
x}u_{\alpha}(y),\quad x\in\X,
\end{equation}
where $\ilim\limits_{\small \alpha\uparrow 1, \ y\to
x}u_{\alpha}(y)$ is the least upper bound of the set of all
$\lambda\in \R_+$ such that there exist $\beta\in [0,1)$ and a
neighborhood $U(x)$ of $x$ such that $\lambda\le
\inf\{u_\alpha(y):\, \alpha\in[\beta,1), y\in U(x)\cap\X\}.$

Also define the following nonnegative functions on $\X$:
\begin{equation}\label{eq:defuuU}
U_\beta(x)  =  \inf\limits_{\small \alpha\in \left[\beta,1
\right)}u_{\alpha}(x), \quad \underline{u}_\beta(x)  =
\ilim\limits_{y\to x}U_\beta(y),\ \quad\quad \beta\in [0,1),\,\,
x\in \X.
\end{equation}
 Observe that all the three defined functions take finite
values at $x\in \X.$ Indeed,
\begin{equation}\label{eqnew9}
\underline{u}_\beta(x)\le U_\beta(x)\le \sup_{\beta\in [0,1)
}\inf\limits_{\small \alpha\in [\beta,1 )}
u_{\alpha}(x)=\ilim\limits_{\alpha\uparrow
1}u_{\alpha}(x)<\infty,\quad \beta\in [0,1), \ x\in\X,
\end{equation}
where the first two inequalities follow from the definitions of
$\underline{u}_\beta$ and $U_\beta$ respectively, and the last
inequality follows from Assumption (${\rm \bf \underline{B}}$).
For $x\in \X$
\begin{equation}\begin{split}\label{eq:ubounded}
u({x})=\sup\limits_{\small \beta\in [0,1),\ R>0}\left[
\inf\limits_{\small \alpha\in \left[\beta,1 \right), \ y\in
B_{R}(x)}u_{\alpha}(y)\right]=\sup\limits_{\beta\in [0,1)}\
\sup\limits_{R>0}\ \inf\limits_{y\in B_R(x)} \
\inf\limits_{\alpha\in \left[\beta,1 \right)}u_{\alpha}(y)\\ =
\sup\limits_{\beta\in [0,1)}\ \sup\limits_{R>0}\ \inf\limits_{y\in
B_R(x)} \ U_{\beta}(y) =\sup\limits_{\beta\in [0,1)}\
\ilim\limits_{y\to x}\ U_{\beta}(y)=\sup\limits_{\beta\in
[0,1)}\underline{u}_{\beta}(x)<\infty, \end{split}
\end{equation}
where $B_R(x)=\{y\in \X \,: \,    \rho(y,x)<R\}$, the first
equality is (\ref{eq131}), the second equality follows from the
properties of infinums, the third and the fifth equalities follow
from (\ref{eq:defuuU}), the fourth equality follows from the
definition of $\limsup$,  and the inequality follows from
(\ref{eqnew9}).  In view of (\ref{eq:defuuU}), the functions
$U_\beta(x)$ and $\underline{u}_\beta(x)$ are nondecreasing in
$\beta$.  Therefore, in view of (\ref{eq:ubounded}),
\begin{equation}\label{eq5.5}
u(x)=\lim\limits_{\beta\uparrow
1}\underline{u}_\beta(x),\qquad\qquad x\in\X.
\end{equation}
 We
also set for $u$ from (\ref{eq5.5})
\begin{equation}\label{defsetA*}
 A^*(x):=\left\{a\in A(x)\, : \,
\overline{w}+u(x)\ge c(x,a)+\int_\X  u(y)q(dy|x,a) \right\}, \
x\in\X,
\end{equation}
and let $A_*(x)$, $x\in\X$, be the sets defined in (\ref{eq:lm000})
for this function $u$; $A_*(x)\subseteq A^*(x).$
\begin{theorem}\label{teor1}
Suppose Assumptions~(${\rm \bf W^*}$) and (${\rm \bf
\underline{B}}$) hold. There exist a stationary policy $\phi$
satisfying (\ref{eq7111}) with $u$ defined in (\ref{eq131}). Thus,
equalities (\ref{eq:7121}) hold for this policy $\phi.$
Furthermore, the following statements hold:

\item{\rm{({\bf a})}} the function $u:\X\to \R_+$, defined in
(\ref{eq131}), is lower semi-continuous;

\item{\rm{({\bf b})}} the nonempty sets $A^*(x)$, $x\in\X$,
satisfy the following properties:

${\rm\bf(b_1)}$ the graph ${\rm Gr}(A^*)=\{(x,a):\, x\in\X, a\in
A^*(x)\}$ is a Borel subset of $\X\times \mathbb{A}$;

${\rm\bf(b_2)}$ for each $x\in\X$ the set $A^*(x)$ is compact;
\item{\rm{({\bf c})}} a stationary policy $\phi$ is optimal for
average costs and satisfies (\ref{eq7111}) with $u$ defined in
(\ref{eq131}), if $\phi(x)\in A^*(x)$ for all $x\in \X$;
 \item{\rm{({\bf d})}}  there exists a stationary policy $\phi$ with
 $\phi(x)\in A_*(x)\subseteq
A^*(x)$ for all $x\in\X$;
 \item{\rm{({\bf e})}}
if, in addition, Assumption~(${\rm \bf Wu}$) holds, then the
function $u$, defined in (\ref{eq131}), is inf-compact. 
\end{theorem}

Before the proof of Theorem~\ref{teor1}, we establish some
auxiliary facts.

\begin{lemma}\label{lemma1}
Under Assumption~(${\rm \bf \underline{B}}$),  the functions $u,
\underline{u}_\alpha:\X\to \mathbb{R}_+,$ $\alpha\in [0,1),$ are
lower semi-continuous on $\X$. If additionally Assumption~(${\rm
\bf W^*}$) holds,  the functions $u_\alpha:\X\to \mathbb{R}_+,$
$\alpha\in [0,1),$ are lower semi-continuous on $\X$. Under
Assumptions~(${\rm \bf Wu}$)   and (${\rm \bf \underline{B}}$),
 the functions $u,u_\alpha, \underline{u}_\alpha:\X\to
\mathbb{R}_+,$ $\alpha\in [0,1),$ are inf-compact on $\X$.
\end{lemma}
\begin{proof}
Since $\underline{u}_\alpha(x)\ge 0$, $\alpha\in[0,1)$ and $x\in
\X$, the functions $\underline{u}_\alpha$, $\alpha\in [0,1),$ are
lower semi-continuous; Feinberg and  Lewis \cite[Lemma 3.1]{FL}.
Since supremum over any set of lower semi-continuous functions is
a lower semi-continuous function, the function $u$ is lower
semi-continuous.

According to (\ref{eq:schal}), $
\overline{w}:=\slim\limits_{\alpha\uparrow1}(1-\alpha)m_{\alpha}=\inf\limits_{\alpha\in
(0,1)}\sup\limits_{\alpha\in[\alpha,1)}(1-\alpha)m_{\alpha}<\infty.
$ Thus, there exists $\alpha_0\in [0,1)$ such that
\begin{equation}\label{eq:estlambda*} \lambda':=
\sup\limits_{\small \alpha\in\left [\alpha_0
,1\right)}(1-\alpha)m_{\alpha}<\infty. \end{equation}

Let us assume that the function $c$ is bounded below.  As
explained in the proof of Lemma~\ref{l5.1EF}, without loss of
generality we can assume that $c\ge 0.$  Then $m_\alpha$ is a
nonnegative, nondecreasing function.  Thus, $(1-\alpha)m_\alpha\le
(1-\alpha) m_{\alpha_0}\le\lambda'/(1-\alpha_0),$ $\alpha\in
[0,\alpha_0)$, and (\ref{eq:estlambda*}) implies that
\begin{equation}\label{eq:estlambda*1}
\lambda^*=\sup\limits_{\alpha\in\left [0
,1\right)}(1-\alpha)m_{\alpha}<\infty. \end{equation}

 According to Theorem~\ref{prop:dcoe}(i, iv,v), under Assumption~(${\rm
\bf W^*}$),  the function $u_\alpha(x)=v_\alpha(x)-m_\alpha$ is
lower semi-continuous, and   a stationary policy $\phi_\alpha$ is
$\alpha$-discount optimal  if and only if for all $x\in \X$
\begin{equation}\label{eq*}
v_{\alpha}(x)=\min\limits_{a\in A(x)}\left\{c(x,a)+\alpha \int_\X
v_{\alpha}(y)q(dy|x,a)\right\}= c(x,\phi_{\alpha}(x))+\alpha
\int_\X v_{\alpha}(y)q(dy|x,\phi_{\alpha}(x)).
\end{equation}
The first equality in (\ref{eq*}) is equivalent to
\begin{equation}\label{eq12}
(1-\alpha)m_{\alpha}+u_{\alpha}(x)=\min\limits_{a\in
A(x)}\left[c(x,a)+\alpha \int_\X  u_{\alpha}(y)q(dy|x,a)\right],
\quad x\in \X.
\end{equation}

Let Assumption~(${\rm \bf Wu}$) hold.  The function
$u_\alpha(x)=v_\alpha(x)-m_\alpha$ is inf-compact by
Theorem~\ref{prop:dcoe}(vi). Consider an arbitrary $\lambda\in
\mathbb{R}_+.$ Since $u(x)\ge \underline{u}_{\alpha_1}(x)\ge
\underline{u}_{\alpha_2}(x)$, $x\in\X$,  for all
$\alpha_1,\alpha_2\in [0,1)$, $\alpha_1\ge\alpha_2$, then
$\mathcal{D}_{{u}}(\lambda)\subseteq\mathcal{D}_{\underline{u}_\alpha}(\lambda)
\subseteq\mathcal{D}_{\underline{u}_0}(\lambda)$, $\alpha\in [0,1).$
Since the functions $u$ and $\underline{u}_\alpha$ are lower
semi-continuous, the sets $\mathcal{D}_{{u}}(\lambda)$ and
$\mathcal{D}_{\underline{u}_\alpha}(\lambda)$ are closed, $\alpha\in
[0,1).$ Therefore, if the set $
\mathcal{D}_{\underline{u}_0}(\lambda)$ is compact then those sets
are also compact and the  functions $u$ and $\underline{u}_\alpha$,
$\alpha\in [0,1)$,  are inf-compact. 

Observe that (\ref{eq:estlambda*1}) and (\ref{eq12}) imply that
$u_\alpha(x)\ge v_1(x)-\lambda^*,$ $x\in X,$ for all $\alpha\in
[0,1).$ This implies $U_0(x)\ge v_1(x)-\lambda^*,$ $x\in X.$ Since
$\underline{u}_0$ is the largest lower-semicontinuous function that
is  less than or equal to $U_0$ at all $x\in\X$, we have
$\underline{u}_0(x)\ge v_1(x)-\lambda^*,$ $x\in X.$  Since the
function $\underline{u}_0$ is lower semi-continuous, the set
$\mathcal{D}_{\underline{u}_0}(\lambda)$ is closed.  In addition,
$\mathcal{D}_{\underline{u}_0}(\lambda)\subseteq
\mathcal{D}_{v_1}(\lambda+\lambda^*)$, where the set
$\mathcal{D}_{v_1}(\lambda+\lambda^*)$ is compact. Thus, the set
$\mathcal{D}_{\underline{u}_0}(\lambda)$ is compact, and the
functions $u$ and $\underline{u}_\alpha$, $\alpha\in [0,1)$,  are
inf-compact.
\end{proof}

\begin{corollary}\label{cor1} Under
Assumption~(${\rm \bf\underline{B}}$), for every sequence
$\alpha_n\uparrow 1$ as $n\to+\infty$ and for every $x\in \X,$
\[u(x)=\ilim\limits_{n\to+\infty,\ y\to
x}\underline{u}_{\alpha_n}(y).\]
\end{corollary}
\begin{proof} Let $\alpha_n\uparrow 1$ as
$n\to+\infty$, and $x\in \X.$ Similar to (\ref{eq:ubounded})
\begin{align*}
\ilim\limits_{n\to+\infty,\ y\to x}\underline{u}_{\alpha_n}(y)= &
\sup\limits_{n=1,2,\ldots}\ \sup\limits_{R>0}\ \inf\limits_{y\in
B_R(x)} \ \inf\limits_{m\ge n}\underline{u}_{\alpha_m}(y)=
\sup\limits_{n=1,2,\ldots}\ \sup\limits_{R>0}\ \inf\limits_{y\in
B_R(x)} \underline{u}_{\alpha_n}(y) \\= & \sup\limits_{n=1,2\ldots}\
\ilim\limits_{y\to
x}\underline{u}_{\alpha_n}(y)=\lim\limits_{n\to\infty
}\underline{u}_{\alpha_n}(x)=u(x), 
\end{align*}
where the second equality holds because the function
$\underline{u}_\alpha(y)$ is nondecreasing in $\alpha$, the fourth
equality holds because it is lower semi-continuous, and the last
equality follows from (\ref{eq5.5}).
\end{proof}

\begin{lemma}\label{lemma3} Under
 Assumptions~(${\rm \bf W^*}$) and (${\rm \bf
\underline{B}}$), the following inequalities hold
\begin{equation}\label{eq:l31} \overline{w}+u(x)\ge
\min\limits_{a\in A(x)} \left[c(x,a)+\int_\X  u(y)q(dy| x,a)\right]
,\qquad x\in\X.
\end{equation}
\end{lemma}
\begin{proof}  
Let us fix an arbitrary $\varepsilon^*>0$. Since $
\overline{w}=\slim\limits_{\alpha\uparrow 1}(1-\alpha)m_{\alpha},
$ there exists $\alpha_0\in [0,1)$ such that
\begin{equation}\label{eq:l32}
\overline{w} +\varepsilon^* > (1-\alpha)m_{\alpha},
\qquad\quad\qquad\qquad\qquad \alpha\in [\alpha_0,1).
\end{equation}

Our next goal is to prove the inequality

\begin{equation}\label{eq:l33}
\overline{w}+\varepsilon^*+u(x)\ge\min\limits_{a\in A(x)}
\left[c(x,a)+\alpha\int_\X  \underline{u}_\alpha(y)q(dy|
x,a)\right],\qquad x\in\X,\ \alpha\in [\alpha_0,1).
\end{equation}

Indeed, by (\ref{eq12}) and (\ref{eq:l32}) for every
$\alpha,\beta\in [\alpha_0,1)$, such that $\alpha\le \beta$, and
for every $x\in \X$
\[
\overline{w}+\varepsilon^*+u_\beta(x)>
(1-\beta)m_{\beta}+u_{\beta}(x)=\min\limits_{a\in
A(x)}\left[c(x,a)+\beta\int_\X  u_{\beta}(y)q(dy|x,a)\right]\ge
\]
\[
\ge \min\limits_{a\in A(x)}\left[c(x,a)+\alpha\int_\X
U_{\alpha}(y)q(dy|x,a)\right].
\]
As right-hand side does not depend on $\beta\in[\alpha,1)$, we
have for all $x\in\X$ and for all $\alpha\in [\alpha_0,1)$
\[
\overline{w}+\varepsilon^*+U_\alpha(x)=
\inf\limits_{\beta\in[\alpha,1)}\left[\overline{w}+\varepsilon^*+u_\beta(x)
\right]
 \ge\min\limits_{a\in A(x)}\left[c(x,a)+\alpha\int_\X
U_{\alpha}(y)q(dy|x,a)\right]\ge
\]
\[
\ge\min\limits_{a\in A(x)}\left[c(x,a)+\alpha\int_\X
\underline{u}_{\alpha}(y)q(dy|x,a)\right]= \min\limits_{a\in
A(x)}\eta_{\underline{u}_{\alpha}}^{\alpha}(x,a).
\]
By Lemma~\ref{lm1}, the function $x\to \min\limits_{a\in
A(x)}\eta_{\underline{u}_{\alpha}}^{\alpha}(x,a)$ is lower
semi-continuous on $\X$. Thus,
\[
\ilim\limits_{y\to x}\min\limits_{a\in
A(y)}\eta_{\underline{u}_{\alpha}}^{\alpha}(y,a)\ge
\min\limits_{a\in
A(x)}\eta_{\underline{u}_{\alpha}}^{\alpha}(x,a), \qquad x\in\X,\
\alpha\in [0,1).
\]
and, as, by definition (\ref{eq:defuuU}),
$\underline{u}_{\alpha}(x)=\ilim\limits_{y\to x}{U}_{\alpha}(y)$,
we finally obtain
\begin{equation}\label{eqalmweu}
\overline{w}+\varepsilon^*+\underline{u}_\alpha(x)\ge\min\limits_{a\in
A(x)}\eta_{\underline{u}_{\alpha}}^{\alpha}(x,a),\qquad\qquad
 x\in\X, \alpha\in [\alpha_0,1).
\end{equation}
As, by (\ref{eq:defuuU}), $u(x)=\sup\limits_{\alpha\in
[\alpha_0,1)}\underline{u}_\alpha(x)$ for all $x\in\X$,
(\ref{eqalmweu}) yields (\ref{eq:l33}).

To complete the proof of the lemma,  we fix an arbitrary $x\in\X$.
By Lemma~\ref{lm1}, for any $\alpha\in [0,1)$ there exists
$a_\alpha\in A(x)$ such that $ \min\limits_{a\in
A(x)}\eta_{\underline{u}_{\alpha}}^{\alpha}(x,a)=
\eta_{\underline{u}_{\alpha}}^{\alpha}(x,a_{\alpha}).$
 Since $\underline{u}_\alpha\ge 0$, for
$\alpha\in [\alpha_0,1)$ the inequality  (\ref{eq:l33}) can be
continued as
\begin{equation}\label{eq:5.14}
\overline{w}+\varepsilon^*+u(x)\ge
\eta_{\underline{u}_{\alpha}}^{\alpha}(x,a_{\alpha})\ge
c(x,a_{\alpha}).
\end{equation}
Thus, for all $\alpha\in [\alpha_0,1)$
\[
a_\alpha\in
\mathcal{D}_{\eta_{\underline{u}_{\alpha}}^{\alpha}(x,\cdot)}(\overline{w}+\varepsilon^*+u(x))\subseteq
\mathcal{D}_{c(x,\cdot)}(\overline{w}+\varepsilon^*+u(x))\subseteq
A(x).
\]
By Lemma~\ref{lm00101}, the set
$\mathcal{D}_{c(x,\cdot)}(\overline{w}+\varepsilon^*+u(x))$ is
compact. Thus, for every sequence $\beta_n\uparrow 1$ of numbers
from $[\alpha_0,1)$ there is a  subsequence $\{\alpha_n\}_{n\ge
1}$ such that the sequence $\{a_{\alpha_n}\}_{n\ge 1}$ converges
and $a_*:=\lim_{n\to\infty} a_{\alpha_n}\in A(x)$.

Consider a sequence $\alpha_n\uparrow 1$ such that
$a_{\alpha_n}\to a_*$ for some $a_*\in A(x).$
 Due to Lemmas~\ref{lemma2} and Corollary~\ref{cor1},
\begin{equation}\label{eq:5.15} \ilim\limits_{n\to +\infty}
\alpha_{n}\int_\X
\underline{u}_{\alpha_n}(y)q(dy|x,a_{n})\ge
\int_\X u(y)q(dy|x,a_*).
\end{equation}

Since the function $c$ is lower semi-continuous, (\ref{eq:5.14})
and (\ref{eq:5.15}) imply
\[\overline{w}+\varepsilon^*+u(x)\ge \limsup\limits_{n\to\infty}
\eta_{\underline{u}_{\alpha_n}}^{\alpha_n}(x,a_{\alpha_n})\ge
c(x,a_*)+\int_\X u(y)q(dy|x,a_*)\ge\min_{a\in A(x)} \eta_u^1(x,a).
\]
Since $\overline{w}+\varepsilon^*+u(x)\ge\min_{a\in A(x)}
\eta_u^1(x,a)$ for any $\varepsilon^*>0$, this is also true when
$\varepsilon^*=0$.
\end{proof}

\begin{proof}[Proof of Theorem~\ref{teor1}] Lemma~\ref{lemma1} contains statements
{\rm ({\bf a})} and {\rm{({\bf e})}}. Since ${\rm
Gr}(A^*)=\{(x,a)\in {\rm Gr}(A):\, g(x,a)\ge 0\}$, where
$g(x,a)=\overline{w}+u(x)- c(x,a)-\int_\X  u(y)q(dy|x,a)$ is a Borel
function, the set ${\rm Gr}(A^*) $ is Borel. The sets $A^*(x)$,
$x\in\X$, are compact in view of Lemma~\ref{lm00101}(b). Thus, the
statement {\rm{({\bf b})}} is proved.  The Arsenin-Kunugui theorem
implies the existence of a stationary policy $\phi$ such that
$\phi(x)\in A^*(x)$ for all $x\in\X.$ Statement {\rm{({\bf e})}}
follows from Lemma~\ref{lm1} and the Arsenin-Kunugui theorem. The
rest follows from Theorem~\ref{Prop1}.
\end{proof}

\begin{theorem}\label{teor3}
Suppose Assumptions~(${\rm \bf W^*}$) and
 (${\rm \bf B}$) hold. Then all the conclusions of
 Theorem~\ref{teor1} hold and, in addition, for a stationary policy $\phi$ satisfying (\ref{eq7111}) with $u$ defined in
 (\ref{eq131}),
\begin{equation} \label{eq5.16}
w^\phi(x)=\underline{w}=\lim\limits_{\alpha\uparrow
1}(1-\alpha)v_\alpha(x)=\lim\limits_{N\to\infty}\frac{1}{N}v^\phi_N(x),\qquad
x\in\X.
\end{equation}
\end{theorem}
\begin{proof}
Consider a sequence $\{\alpha(n)\}_{n\ge1}$  such that $
\alpha(n)\uparrow 1$ as $ n\to+\infty,$ and
\[\lim\limits_{n\to+\infty}(1-\alpha(n))m_{\alpha(n)}=\underline{w}.
\]

Define the following nonnegative functions on $\X$:
\[
{\tilde U}_n(x)  = \inf\limits_{m\ge n}u_{\alpha(m)}(x),\
\underline{{\tilde u}}_n(x) = \ilim\limits_{y\to x}{\tilde U}_n(y),
\quad n\ge 1,\,\, x\in \mathbb{X},
\]
and
\begin{equation}\label{eq:deftildeu}
{\tilde u}(x)=\sup\limits_{n\ge 1}\underline{{\tilde
u}}_n(x),\,\,x\in \mathbb{X}.
\end{equation}
Observe that
\begin{equation}\label{eq5.18E}\underline{{\tilde u}}_n(x)\le {\tilde U}_n(x)\le
\slim_{m\to+\infty }u_{\alpha(m)}(x)<\infty,\quad x\in\X,\
n=1,2,\ldots,\end{equation} where the first two inequalities follow
from the definitions of $\underline{{\tilde u}}_n$ and ${\tilde
U}_n$ respectively, and the last inequality follows from
Assumption~(${\rm \bf B}$). As follows from (\ref{eq:deftildeu}) and
(\ref{eq5.18E}), ${\tilde u}(x)\le\slim_{m\to+\infty
}u_{\alpha(m)}(x)<+\infty$. According to Feinberg and Lewis
\cite[Lemma 3.1]{FL}, the functions $\underline{{\tilde u}}_n$,
$n\ge 1,$ are lower semi-continuous on $\X$. Therefore, their
supremum $\tilde u$ is also lower semi-continuous. In addition,
\[
{\tilde u}(x)= \sup\limits_{n\ge 1}\ \sup\limits_{R>0}\
\inf\limits_{y\in B_R(x)} \ \inf\limits_{m\ge n}u_{\alpha_m}(y)=
\ilim\limits_{n\to +\infty,\, y\to x }u_{\alpha(n)}(y),\quad x\in
\mathbb{X},
\]
where the first equality follows from the definitions of ${\tilde
U}_n$, $\underline{{\tilde u}}_n,$ and $\tilde u$, and the second
equality is the definition of the $\liminf$. Since ${\tilde
U}_n(x)\uparrow $, we have $\underline{{\tilde u}}_n(x)\uparrow
{\tilde u}(x)$ as $n\to\infty$ for all $x\in\X.$

We  show next that for each $x\in\X$
\begin{equation}\label{eq141}
\underline{w}+{\tilde u}(x)\ge \inf\limits_{a\in
A(x)}\left[c(x,a)+\int_\X {\tilde u}(y)q(dy|x,a)\right].
\end{equation}
Indeed let us fix any $\varepsilon^*>0$. By the definition of
$\underline{w}$, there exists a subsequence $\{\alpha(n_k)\}_{k\ge
1}\subseteq \{\alpha(n)\}_{n\ge 1}$ such that for  $k=1,2,\ldots$
\[
\underline{w} +{\varepsilon^*}\ge (1-\alpha(n_k))m_{\alpha(n_k)}.
\]
Let $x\in \mathbb{X}$ be an arbitrary state. By
Theorem~\ref{prop:dcoe} for each $k\ge 1$ there exists $a_{n_k}\in
A_{\alpha(n_k)}(x)$ such that
$$
(1-\alpha({n_k}))m_{\alpha({n_k})}+{
u}_{\alpha({n_k})}(x)=c(x,a_{n_k})+\alpha({n_k})\int_\X {
u}_{\alpha({n_k})}(y)q(dy|x,a_{n_k}).
$$
Thus, similarly to the proof of Lemma~\ref{lemma3}, we get
(\ref{eq141}).

From Lemma~\ref{lm1} and the Arsenin-Kunugui theorem there exists
a stationary policy ${\tilde \phi}\in\mathbb{F}$ such that for any
$x\in \X$ \begin{equation}\label{eq5.20E} \underline{w}+{\tilde
u}(x)\ge c(x, {\tilde \phi}(x))+\int_\X {\tilde u}(y)q(dy|x,
{\tilde \phi}(x)).
\end{equation}
Thus, by Sch$\ddot{\rm {a}}$l \cite[Proposition~1.3]{Schal}
described in (\ref{eq7}), for all $x\in \X$
\begin{equation}\label{999AB}
\overline{w}=\underline{w}=w(x)=w^{{\tilde
\phi}}(x)=\lim\limits_{\alpha\uparrow
1}(1-\alpha)v_{\alpha}(x)=w^*. 
\end{equation}

 Let us choose any stationary policy $\phi$ such
that inequalities (\ref{eq7}) and (\ref{eq7111}) hold with the
function $u$ defined in (\ref{eq131}). Since
$\overline{w}=\underline{w} ,$ according to Theorem~\ref{teor1},
 such a stationary policy exists. Theorem~\ref{Prop1}  implies that the
stationary policy $\phi$ satisfies (\ref{eq:7121}), and Sch\"al
\cite[Proposition~1.3]{Schal} (see (\ref{eq7})) implies that
(\ref{999AB}) holds with ${\tilde \phi}=\phi$.
%
%

  In addition, (\ref{999AB}) with ${\tilde \phi}=\phi$ implies that for all
$x\in \X$
\[ w^\phi(x)=\lim\limits_{\alpha\uparrow 1}(1-\alpha) m_\alpha=
\lim\limits_{\alpha\uparrow
1}(1-\alpha)(v_\alpha(x)-u_\alpha(x))=\lim\limits_{\alpha\uparrow
1}(1-\alpha) v_\alpha(x),
\]
where the last equality follows from Assumption~(${\rm \bf B}$).
Thus, for all $x\in \X$
\begin{align*}
w^\phi(x) & =
\slim\limits_{n\to\infty}\frac{1}{n}v^\phi_n(x)\ge\slim\limits_{\alpha\uparrow
1}(1-\alpha) v^\phi_\alpha(x) \ge \ilim\limits_{\alpha\uparrow
1}(1-\alpha) v^\phi_\alpha(x)\\ & \ge \lim\limits_{\alpha\uparrow
1}(1-\alpha) v_\alpha(x)=w^\phi(x),
\end{align*}
where the first inequality follows from the Tauberian theorem (see
Sennott~\cite[Section A.4]{Senb} or \cite[Proposition 5.7]{Sens}),
and the last inequality follows from $v_\alpha^\phi(x)\ge
v_\alpha(x)$ and the existence of the limit.
So, we have, the existence of 
$\lim\limits_{\alpha\uparrow 1}(1-\alpha) v^\phi_\alpha(x).$ Thus,
 the Karamata Tauberian theorem (Sennott~\cite[Section A.4]{Senb}
or \cite[Proposition 5.7]{Sens}) implies
$w^\phi(x)=\lim_{n\to\infty}\frac{1}{n}v_n^\phi(x).$
\end{proof}

\begin{corollary}
Under Assumptions~(${\rm \bf W^*}$) and (${\rm \bf
\underline{B}}$), the conclusions of Theorems~\ref{teor1} and
\ref{teor3} remain correct, if the function $u$ is substituted
with the function $\tilde u$ defined in (\ref{eq:deftildeu}).
\end{corollary}
\begin{proof}
As shown in the proof of Theorem~\ref{teor3}, there exists a
stationary policy $\tilde\phi$ satisfying (\ref{eq5.20E}).  The
function $\tilde u$ is nonnegative, lower semi-continuous, and takes
finite values.  Thus, both \cite[Proposition~1.3]{Schal} (see
(\ref{eq7})) and Theorem~\ref{Prop1} can be applied to this
function. The proof of statements (a)--(d) of Theorem~\ref{teor1}
uses just these properties of $u$.  Statement (e) follows from
Lemma~\ref{lemma1}, whose proof remains unchanged if $u$ is replaced
with $\tilde u$.
\end{proof}

\section{Approximation of Average Cost Optimal Strategies by $\alpha$-discount Optimal
Strategies}\label{s6} For a family of sets $\{{\rm
Gr}(A_\alpha)\}_{\alpha\in (0,1)}$, $x\in \X$, considered in
Theorem~\ref{prop:dcoe}, we pay our attention to its upper
topological limit
\[
\mathop{\rm \overline{Lim}}\limits_{\alpha \uparrow 1} {\rm
Gr}(A_\alpha)=\left\{(x,a)\in \X\times\mathbb{A}\, :
\begin{array}{l} \exists \alpha_{n}\uparrow1, \, n\to
+\infty, \ \exists (x_n,a_n) \in {\rm Gr}(A_{\alpha_n}), \, n\ge
1,\\
\mbox{ such that } (x,a)=\lim\limits_{n\to+\infty} (x_n,a_{n})
\end{array}
\right\},
\]
defined, for example, in Zgurovsky et al.~\cite[Chapter~1,
p.~3]{ZMK1}. Let us set
\[
 A^{app}(x):=\left\{a\in A^*(x)\, : \,
(x,a)\in \mathop{\rm \overline{Lim}}\limits_{\alpha \uparrow 1}
{\rm Gr}(A_\alpha) \right\}, \qquad x\in\X.
\]

\begin{theorem}\label{teorapp}
Under Assumptions~(${\rm \bf W^*}$) and (${\rm \bf
\underline{B}}$), the graph ${\rm Gr}(A^{app})$ is a Borel subset
of ${\rm Gr}(A^*)$, and for each $x\in\X$ the set $A^{app}(x)$ is
nonempty and compact. Furthermore, there exists a stationary
 policy $\phi^{app}$ such that
$\phi^{app}(x)\in A^{app}(x)$ for all $x\in X$, and any such policy is average-cost optimal.

\end{theorem}
\begin{proof}
Let us fix an arbitrary $x\in\X$. From (\ref{eq131}) (the
definition of $u$), there exists $\{y_n,\alpha_n\}_{n\ge
1}\subseteq \X\times(0,1)$ such that $y_n\to x$,
$\alpha_n\uparrow1$, $u_{\alpha_n}(y_n)\to u(x)$, $n\to+\infty$.

Let us choose an arbitrary $\varepsilon^*>0$ and $b_n\in
A_{\alpha_n}(y_n)$, $n\ge 1$. Since
$\overline{w}=\slim\limits_{\alpha\uparrow 1}(1-\alpha)m_{\alpha}$,
there exists $N\ge 1$ such that $ u(x)+\frac{\varepsilon^*}{2} \ge
u_{\alpha_n}(y_{n})$ and $ \overline{w} +\frac{\varepsilon^*}{2} \ge
(1-\alpha_n)m_{\alpha_n}$ for all $n\ge N.$

By definition of the sets $A_\alpha(\cdot)$, for each $n\ge N$
$$
(1-\alpha_n)m_{\alpha_n}+u_{\alpha_n}(y_{n})=c(y_{n},b_{n})+\alpha_n\int_\X
u_{\alpha_n}(y)q(dy|y_{n},b_{n})
=\eta_{u_{\alpha_n}}^{\alpha_n}(y_n,b_n).
$$
Thus, for all $n\ge N$
$$
\overline{w}+\varepsilon^*+u(x)>
\eta_{u_{\alpha_n}}^{\alpha_n}(y_n,b_n)\ge
\eta_{U_{\alpha_n}}^{\alpha_n}(y_n,b_n)\ge\eta_{\underline{u}_{\alpha_n}}^{\alpha_n}(y_n,b_n)\ge
c(y_n, b_n) .
$$
Therefore, because of Assumption~(${\rm \bf W^*}$)(ii), the
sequence $\{b_n\}_{n\ge 1}$ has a subsequence $\{b_{n_k}\}_{k\ge
1}$ such that $b_{n_k}\to a$, as $k\to+\infty$, for some $a\in
A(x)$. Thus, $(x,a)\in \mathop{\rm \overline{Lim}}\limits_{\alpha
\uparrow 1} {\rm Gr}(A_\alpha)$.

Let us prove that $(x,a)\in {\rm Gr}(A^*)$. Indeed, as
$\alpha_{n_k}\underline{u}_{\alpha_{n_k}}(\cdot)\uparrow
u(\cdot)$, $k\to+\infty$, then due to Lemma~\ref{lemma2} and
Corollary~\ref{cor1}, 
\[
\ilim\limits_{k\to +\infty} \alpha_{n_k}\int_\X
\underline{u}_{\alpha_{n_k}}(x)q(dy|y_{n_k},b_{n_k}) \ge \int_\X
u(x)q(dy|x,a).
\]
Thus, by Lemma~\ref{lm1}, $ \overline{w}+\varepsilon^*+u(x)\ge
\eta_{u}^1(x,a), $ and this is true for any $ \varepsilon^*>0 .$
This implies $ \overline{w}+u(x)\ge \eta_{u}^1(x,a). $ This
inequality means that $(x,a)\in {\rm Gr}(A^*)$ and $A^{app}(x)\ne
\emptyset$, since $(x,a)\in \mathop{\rm
\overline{Lim}}\limits_{\alpha \uparrow 1} {\rm Gr}(A_\alpha)$.
The set $A^{app}(x)$ is compact because of the
 closureness of $\mathop{\rm
\overline{Lim}}\limits_{\alpha \uparrow 1} {\rm Gr}(A_\alpha)$ (see
Zgurovsky et al.~\cite[Chapter~1, p.~3]{ZMK1}) and
 Theorem~\ref{teor1}(b).
The second statement of the theorem follows from the
Arsenin-Kunugui theorem.
\end{proof}

\begin{corollary}\label{cor2}
Under Assumptions~(${\rm \bf W^*}$) and (${\rm \bf \underline{B}}$),
for any stationary average-cost optimal policy $\phi^{app}$, such
that $\phi^{app}(x)\in A^{app}(x)$ for all $x\in \X$, for every
$x\in\X$ there exist $\alpha_{n}(x)\uparrow1$ and $y_n(x)\to x$ as
$n\to +\infty$ such that $a_n(x) \in A_{\alpha_n(x)}(y_n(x))$, $n\ge
1,$ and $\phi^{app}(x)=\lim_{n\to+\infty} a_{n}(x)$.
\end{corollary}
\begin{proof}
Following Theorem~\ref{teorapp}, consider a stationary
average-cost optimal policy $\phi^{app}$ such that
$\phi^{app}(x)\in A^{app}(x)$ for all $x\in X$. Furthermore, since
$A^{app}(x)\subseteq A^*(x)$ for all $x\in\X,$ any such a policy
is optimal.
Let us fix an arbitrary $x\in \X$. By definition of $A^{app}(x)$,
we have that $(x,\phi^{app}(x))\in \mathop{\rm
\overline{Lim}}\limits_{\alpha \uparrow 1} {\rm Gr}(A_\alpha)$.
Then, there exist $\alpha_{n}(x)\uparrow1$, $n\to +\infty$, and
$(y_n(x),a_n(x)) \in {\rm Gr}(A_{\alpha_n})$, $n\ge 1$, such that
$(x,\phi^{app}(x))=\lim\limits_{n\to+\infty} (y_n(x),a_{n}(x))$,
i.e. $\phi^{app}(x)=\lim\limits_{n\to+\infty} a_{n}(x)$, where
$a_n(x) \in A_{\alpha_n(x)}(y_n(x))$, $n\ge 1$,
$\alpha_{n}(x)\uparrow1$ and $y_n(x)\to x$ as $n\to +\infty$.
\end{proof}

We remark that, if we replace in (\ref{defsetA*}) the function $u$
with $\tilde u$ defined in (\ref{eq:deftildeu}),
Theorem~\ref{teorapp} and Corollary~\ref{cor2} remain correct.


Let us set
\[
X_\alpha:=\{x\in\X\,: v_\alpha(x)=m_\alpha\},\quad \alpha\in
[0,1).
\]
Under Assumptions~(${\rm \bf G}$), $m_\alpha<\infty$.  If
Assumptions~(${\rm \bf G}$) and (${\rm \bf Wu}$) hold then
Theorem~\ref{prop:dcoe} implies that $X_\alpha$ is a compact set
for each $\alpha\in [0,1)$.  This fact is useful to establish the
validity of Assumptions~(${\rm \bf G}$); see Feinberg and
Lewis~\cite[Lemma 5.1]{FL} and references therein.
\begin{theorem}\label{lemC}
Let Assumptions~(${\rm \bf G}$) and (${\rm \bf Wu}$) hold. Then
there exists a compact set $\mathcal{K}\subseteq\X$ such that
$X_\alpha\subseteq \mathcal{K}$ for each $\alpha\in [0,1)$.
\end{theorem}
\begin{proof}
From Assumption~(${\rm \bf G}$) and Theorem~\ref{prop:dcoe} we
have that for each $\alpha\in [0,1)$
\[
\emptyset\ne X_\alpha=\{x\in\X\, : \, u_\alpha(x)=0\}=
\mathcal{D}_{u_\alpha}(0)\subseteq\mathcal{D}_{U_\alpha}(0)\subseteq
\mathcal{D}_{{\underline{u}}_\alpha}(0)\subseteq\mathcal{D}_{{\underline{u}}_{0}}(0).
\]
In virtue of Lemma~\ref{lemma1}, we have that
${\underline{u}}_{0}:\X\to[0,+\infty)$ is inf-compact function on
$\X$. Setting $\mathcal{K}= \mathcal{D}_{{\underline{u}}_{0}}(0)$,
we obtain the statement of the theorem.
\end{proof}

\section{Illustrative Example}\label{s7}

%
%

 The following example is from
Hern\'andez-Lerma~\cite{HLerma}.  Let
\[
x_{n+1}=\gamma x_n+\beta a_n+\xi_n,\qquad n=0,1,...,
\]
and
\[
c(x,a)=qx^2+ra^2,
\]
where (a) $q$ and $r$ are positive constants, $\gamma$ and $\beta$
are two constants satisfying $\gamma\beta>0$, and (b) $\xi_n$ are
independent and identically distributed (iid) random variables with
zero mean, finite variance, and continuous density.

This problem is solved in Hern\'andez-Lerma~\cite{HLerma}, where a
stationary average-cost optimal policy is computed. This problem
corresponds to an MDP with $\X=\A=\R$ and with setwise continuous
transition probabilities. However, if $\xi_n$ do not have a
density, the transition probability may not be setwise continuous,
but they are weakly continuous; see Feinberg and Lewis~\cite[p.
48]{FL1} for detail. If $\xi_n$ are arbitrary iid random variables
with zero mean and finite variance, this problem satisfies
Assumption~(${\rm \bf Wu}$) and, similarly to the case when there
are densities, it satisfies Assumption~(${\rm \bf B}$). Thus,
Theorem~\ref{teor3} can be applied. The optimal policy provided in
Hern\'andez-Lerma~\cite{HLerma} is also optimal when $\xi_n$ may
not have a density.

\appendix \section{Proof of Lemma~\ref{lemma2}}

\begin{proof} First, we prove the lemma for uniformly bounded above
functions $h_n$.  Let  $h_n(s)\le K<\infty$ for all $n= 1,2,...$
and all $s\in S$. For $n=1,2,\ldots$ and $s\in S$, define
\[H_n(s)=\inf\limits_{m\ge n} h_m(s)\quad {\rm and}\quad
\underline{h}_n(s)=\ilim\limits_{s'\to s}H_n(s').
\]
The functions $\underline{h}_n:S\to [0,+\infty)$, $n=1,2,\ldots,$
are lower semi-continuous; see, for example, Feinberg and Lewis
\cite[Lemma~3.1]{FL}).  In addition, for $s\in S$

\begin{equation}\label{equkb0}\underline{h}_n(s)\downarrow \underline{h}(s)\qquad {\rm
as\quad
}n\to\infty.
\end{equation}
Weak convergence of $\{\mu_n\}_{n\ge 1}$ to $\mu$ is equivalent to
\begin{equation}\label{eq15}
\ilim\limits_{n\to +\infty}\mu_n(A)\ge\mu(A)\qquad {\rm for\ all}\
A\in\mathcal{O},
\end{equation}
where $\mathcal {O}$ is the family of all open subsets of the
space $S;$ Billingsley \cite[Theorem 2.1]{Bil}.

Fix an arbitrary $t>0.$ By (\ref{equkb0}), if $h(s)>t$ then
$\underline{h}_n(s)>t$, $ n= 1,2,\ldots,$ and
\begin{equation}\label{appunion}
\left\{s\in S\,:\, h(s)>t\right\}=\bigcup\limits_{n\ge 1}S_{n},
\end{equation}
where
\[
S_{n}=\{s\in S\,:\, \underline{h}_n(s)>t\},\quad n= 1,2,\ldots,
\]
are open sets, since the functions $\underline{h}_n: S\to
\mathbb{R}_+$ are lower semi-continuous. In addition,
\begin{equation}\label{apmonot}
S_{n}\subseteq S_{n+1},\qquad n=1,2,\dots\ .
\end{equation}
Thus, 

\[
\mu(\{s\in S\,:\, h(s)> t\})=\lim\limits_{n\to +\infty}\mu(S_{n})\le
\lim\limits_{n\to +\infty}\ilim\limits_{m\to
+\infty}\mu_m(S_{n})
\]
\[
\le \slim\limits_{n\to +\infty}\ \ilim\limits_{m\to
+\infty}\mu_m(S_{m})= \ilim\limits_{n\to
+\infty}\mu_n(S_{n})=\ilim\limits_{n\to +\infty}\mu_n (\{s\in S\,:\,
\underline{h}_n(s)> t\}),
\]
where the first equality follows from (\ref{apmonot}) and
(\ref{appunion}), the first inequality follows from to
(\ref{eq15}), and the second inequality follows from
(\ref{apmonot}).

Thus  Serfozo \cite[Lemma 2.1]{Serfozo} yields
\[
\int_S \underline{h}(s)\mu(ds)\le \ilim\limits_{n\to
+\infty}\int_S \underline{h}_n(s)\mu_n(ds)\le \ilim\limits_{n\to
+\infty}\int_S h_n(s)\mu_n(ds),
\]
where the  second inequality is fulfilled due to
\[
\underline{h}_n(s)\le H_n(s)\le h_n(s),\qquad  s\in S, \  n=
1,2,\ldots\ .
\]

\textbf{Case~2.} Consider a sequence $\{h_n\}_{n\ge 1}$ of
measurable nonnegative $\overline{\mathbb{R}}$-valued functions on
$S$. For  $\lambda>0$  set
$h_n^{\lambda}(s):=\min\{h_n(s),\lambda\}$, $s\in S$, $n=
1,2,\ldots\ $. Since the functions $h_n^{\lambda}$ are uniformly
bounded above,
\[
\int_S \underline{h}^{\lambda}(s)\mu(ds)\le \ilim\limits_{n\to
+\infty}\int_S h_n^{\lambda}(s)\mu_n(ds)\le \ilim\limits_{n\to
+\infty}\int_S h_n(s)\mu_n(ds),
\]
where $\underline{h}^{\lambda}(s)=\ilim\limits_{n\to+\infty,\, s'\to
s}h_n^{\lambda}(s')$, $\lambda>0$, $s\in S$.

Then, using Fatou's lemma,
\[
\int_S \underline{h}(s) \mu(ds)\le
\ilim\limits_{\lambda\to+\infty} \int_S \underline{h}^{\lambda}(s)
\mu(ds).
\] \noindent
\end{proof}

\vspace{.3cm}
 {\bf Acknowledgements.}
 Research of the first
author was partially supported by NSF grant  CMMI-0900206. The
authors thank Professor M.Z.~Zgurovsky for initiating their
research cooperation.

\end{document}